\documentclass[fleqn,11pt,twoside]{article}

\usepackage[english]{babel}
\usepackage[utf8]{inputenc}
\usepackage[T1]{fontenc}

\usepackage{amsmath}
\usepackage{amssymb}
\usepackage{mathtools}
\usepackage{amsthm}
\usepackage{xcolor}
\usepackage{epsfig}
\usepackage{subfigure}
\usepackage{graphicx}
\usepackage{latexsym}
\usepackage{lscape}
\usepackage{physics}
\usepackage{braket}
\usepackage{cleveref}
\usepackage{tikz}
\usetikzlibrary{calc}

\newcommand{\Pj}{\mathbb{CP}}

\newcommand{\R}{\mathbb{R}}

\newcommand{\Cp}{\mathbb{C}}

\renewcommand{\vec}[1]{\mathbf{#1}}

\renewcommand{\epsilon}{\varepsilon}
\renewcommand{\imath}{\mathrm{i}}

\DeclareMathOperator{\ud}{d}

\allowdisplaybreaks
\makeatletter 
\renewcommand{\pdv}[2]{\begingroup 
\@tempswafalse\toks@={}\count@=\z@ 
\@for\next:=#2\do 
{\expandafter\check@var\next\@nil
 \advance\count@\der@exp 
 \if@tempswa 
   \toks@=\expandafter{\the\toks@\,}% 
 \else 
   \@tempswatrue 
 \fi 
 \toks@=\expandafter{\the\expandafter\toks@\expandafter\partial\der@var}}% 
\frac{\partial\ifnum\count@=\@ne\else^{\number\count@}\fi#1}{\the\toks@}% 
\endgroup} 
\def\check@var{\@ifstar{\mult@var}{\one@var}} 
\def\mult@var#1#2\@nil{\def\der@var{#2^{#1}}\def\der@exp{#1}} 
\def\one@var#1\@nil{\def\der@var{#1}\chardef\der@exp\@ne} 

\makeatletter
\newcommand{\copyrightnote}[2]{{\renewcommand{\thefootnote}{}
 \footnotetext{\small\it
\begin{flushleft}
 \copyright \ #1   #2  
\end{flushleft}}}}

\newcommand{\Name}[1]{\begin{flushleft}
                       \LARGE \bf #1
                       \end{flushleft}\vspace{-3mm}}

\newcommand{\Author}[1]{\begin{flushleft}
                       \it #1 \end{flushleft}}

\newcommand{\Address}[1]{\begin{flushleft}
                       \it #1 \end{flushleft}}

\newcommand{\Date}[1]{\begin{flushleft}
                      \small  \it #1 \end{flushleft}}

\newcommand{\evenhead}{Author \ name}
\newcommand{\oddhead}{Article \ name}

\renewcommand{\@evenhead}{
\hspace*{-3pt}\raisebox{-15pt}[\headheight][0pt]{\vbox{\hbox to \textwidth
{\thepage \hfil \evenhead}\vskip4pt \hrule}}}
\renewcommand{\@oddhead}{
\hspace*{-3pt}\raisebox{-15pt}[\headheight][0pt]{\vbox{\hbox to \textwidth
{\oddhead \hfil \thepage}\vskip4pt\hrule}}}
\renewcommand{\@evenfoot}{}
\renewcommand{\@oddfoot}{}

\long\def\@makecaption#1#2{%
  \vskip\abovecaptionskip
  \sbox\@tempboxa{\small \textbf{#1.}\ \ #2}%
  \ifdim \wd\@tempboxa >\hsize
    {\small \textbf{#1.}\ \ #2}\par
  \else
    \global \@minipagefalse
    \hb@xt@\hsize{\hfil\box\@tempboxa\hfil}%
  \fi
  \vskip\belowcaptionskip}

\newcommand{\JNMPnumberwithin}[3][\arabic]{%
  \@ifundefined{c@#2}{\@nocounterr{#2}}{%
    \@ifundefined{c@#3}{\@nocnterr{#3}}{%
      \@addtoreset{#2}{#3}%
      \@xp\xdef\csname the#2\endcsname{%
        \@xp\@nx\csname the#3\endcsname .\@nx#1{#2}}}}%
}

\renewenvironment{proof}[1][\proofname]{\par
  \normalfont
  \topsep6\p@\@plus6\p@ \trivlist
  \item[\hskip\labelsep\textbf{%
    #1\@addpunct{.}}]\ignorespaces
}{%
  \qed\endtrivlist
}

\newcommand{\resetfootnoterule} {
  \renewcommand\footnoterule{%
  \kern-3\p@
  \hrule\@width.4\columnwidth
  \kern2.6\p@}
}

\renewcommand{\footnoterule}{}

\makeatother

\theoremstyle{plain}
\newtheorem{theorem}{Theorem}[section]

\newtheorem{corollary}[theorem]{Corollary}
\newtheorem{lemma}[theorem]{Lemma}
\theoremstyle{definition}
\newtheorem{definition}[theorem]{Definition}

\theoremstyle{remark}
\newtheorem{remark}[theorem]{Remark}
\theoremstyle{remark}

\newtheorem*{conjecture*}{Conjecture}

\usepackage[backend=biber,sorting=nyt,maxnames=99,giveninits=true,isbn=false,doi=false,url=false,citestyle=numeric-comp]{biblatex}
\begin{document}

\renewcommand{\evenhead}{ {\LARGE\textcolor{blue!10!black!40!green}{{\sf \
\ \ ]ocnmp[}}}\strut\hfill G Gubbiotti, D McLare, and G R W Quispel}
\renewcommand{\oddhead}{ {\LARGE\textcolor{blue!10!black!40!green}{{\sf
]ocnmp[}}}\ \ \ \ \ Elementary construction of the invariants of 2D Kahan maps}

%%%% Matter for the first page 
\thispagestyle{empty}
\newcommand{\FistPageHead}[3]{
\begin{flushleft}
\raisebox{8mm}[0pt][0pt]
{\footnotesize \sf
\parbox{150mm}{{Open Communications in Nonlinear Mathematical Physics}\ \ \ \ {\LARGE\textcolor{blue!10!black!40!green}{]ocnmp[}}
\quad Special Issue 1, 2024\ \  pp
#2\hfill {\sc #3}}}\vspace{-13mm}
\end{flushleft}}

\FistPageHead{1}{\pageref{firstpage}--\pageref{lastpage}}{ \ \ }

\strut\hfill

\strut\hfill

\copyrightnote{The author(s). Distributed under a Creative Commons Attribution 4.0 International License}

\begin{center}
%{\Large  {\sf This article is part of a Special Issue in Memory of Professor Decio Levi}}
{  {\bf This article is part of an OCNMP Special Issue\\ 
\smallskip
in Memory of Professor Decio Levi}}
\end{center}

\smallskip

\Name{An elementary construction of modified Hamiltonians and modified measures of 2D Kahan maps}

\Author{Giorgio Gubbiotti}

\Address{Dipartimento di Matematica ``Federigo Enriques'', Universit\`a degli
Studi di Milano, Via C. Saldini 50, 20133 Milano, Italy \& INFN Sezione Milano,
Via G. Celoria 16, 20133 Milano, Italy,
\\
email: giorgio.gubbiotti\@unimi.it}

\Author{David I. McLaren}
\Author{G.~R.~W. Quispel}
\Address{Department of Mathematical and
    Physical Sciences, La Trobe University, Bundoora,VIC 3083, Australia,
    \\
    email: d.mclaren\@latrobe.edu.au
\\
email: r.quispel@latrobe.edu.au}

\Date{Received September 9, 2023; Accepted February 22, 2024}

\setcounter{equation}{0}

\begin{abstract}
    \noindent
    We show how to construct in an elementary way the invariant of the
    KHK discretisation of a cubic Hamiltonian system in two dimensions.
    That is, we show that this invariant is expressible as the product
    of the ratios of affine polynomials defining the prolongation of
    the three parallel sides of a hexagon. On the vertices of such a
    hexagon lie the indeterminacy points of the KHK map. This result
    is obtained by analysing the structure of the singular fibres of the
    known invariant. We apply this construction to several examples,
    and we prove that a similar result holds true for a case outside
    the hypotheses of the main theorem, leading us to conjecture that
    further extensions are possible.
\end{abstract}

\label{firstpage}

%%%% The Article text starts here
\section{Introduction}

In recent years a lot of interest has arisen regarding the problem of finding
good discretisations of continuous systems.  By good discretisation, here we
mean a discretisation which preserves the properties of its continuous
counterpart as much as possible.  Within this framework a procedure called
\emph{Kahan-Hirota-Kimura (KHK) discretisation} became popular. This
discretisation method, defined for quadratic ordinary differential equations
(ODEs), was first discovered by Kahan \cite{Kahan1993,KahanLi1997}. It was
rediscovered independently by Hirota and Kimura, who used it to produce
integrable discretisations of the Euler top \cite{HirotaKimura2000} and the
Lagrange top \cite{KimuraHirota2000}. More recently, the KHK method has been
generalised to birational discretisation of ODEs of higher order and/or degree.
These novel methods are called polarisation methods
(see~\cite{McLachlan_etal2023} and references therein). 

The results of Hirota and Kimura attracted the attention of Petrera, Suris,
and collaborators, who extended the work to a significant number of other
integrable quadratic ODEs
\cite{PetreraPfadlerSuris2009,PetreraSuris2010,PetreraPfadlerSuris2011}.
This in turn led to the work of Celledoni, Owren, Quispel, and
collaborators,
\cite{CelledoniMcLachlanOwrenQuispel2013,CelledoniMcLachlanOwrenQuispel2014}
who showed that the KHK method is the restriction of a Runge-Kutta method
to quadratic differential equations. That is, given a quadratic system
\begin{equation}
    \vec{\dot{x}}=\vec{f}\left( \vec{x} \right),
    \quad
    \vec{x}\colon\R\to\R^{N},
    \quad
    \vec{f}\colon\R^{N}\to\R^{N},
    \label{eq:firstord}
\end{equation}
its KHK discretisation is given by the following formula:
\begin{equation}
    \frac{\vec{x'}-\vec{x}}{h} = 
    2\vec{f}\left( \frac{\vec{x'}+\vec{x}}{2} \right)
    -\frac{\vec{f}\left( \vec{x'} \right)+\vec{f}\left( \vec{x}\right)}{2},
    \label{eq:kahan}
\end{equation}
where $\vec{x} = \vec{x}\left( n h \right)$, $\vec{x'}=\vec{x}\left( (n+1)h
\right)$, and $h$ is an infinitesimal parameter, i.e. $h\to0^{+}$.

While the form of eq.\ \eqref{eq:kahan} makes it clear that the KHK method
is the reduction of a Runge-Kutta method (and hence e.g. commutes with
affine coordinate transformations), and it invariant under
$\vec{x}\leftrightarrow\vec{x'}$, $h\leftrightarrow -h$.  It is less clear
that the method is in fact birational, i.e. the map $\vec{x} \rightarrow
\vec{x}'$ and the inverse map $\vec{x}' \rightarrow \vec{x}$ are both
rational functions.  In particular, from eq.\ \eqref{eq:kahan} it is clear
that the inverse map is obtained from the substitution $h\to-h$ in
formula~\eqref{eq:kahan}. These statements becomes more clear in the
equivalent form \cite{CelledoniMcLachlanOwrenQuispel2013}: \begin{equation}
    \vec{x'}= \boldsymbol{\Phi}\left( \vec{x},h \right)= \vec{x} + h \left(
I_{N}-\frac{h}{2} \vec{f'}\left( \vec{x} \right)
\right)^{-1}\vec{f}\left(\vec{x}\right), \label{eq:kahan2} \end{equation}
where $\vec{f'}\left( \vec{x} \right)$ is the Jacobian of the function
$\vec{f}$.  From the above remark, we write the inverse map as
$\boldsymbol{\Phi}^{-1}\left( \vec{x},h \right)=\boldsymbol{\Phi}\left(
\vec{x},-h \right)$.

Some of the properties found in
\cite{PetreraPfadlerSuris2009,PetreraSuris2010,PetreraPfadlerSuris2011}
were later explained in
\cite{CelledoniMcLachlanOwrenQuispel2013,CelledoniMcLachlanOwrenQuispel2014}.
For instance, the following theorem on the existence of invariants in the
case when the system \eqref{eq:firstord} is Hamiltonian. That is, when there
exists a function $H\colon \R^{N}\to\R$, the Hamiltonian, and a constant skew-symmetric
matrix $J\in M_{N,N}\left( \R \right)$ such that:
\begin{equation}
    \vec{\dot{x}}=J\grad H\left( \vec{x} \right).
    \label{eq:firstordham}
\end{equation}
We observe that throughout this paper the capital letter $H$ will
always denote a Hamiltonian function, while the lower-case $h$ will
denote a time step in a discrete-time map.
Then, the statement proven in \cite{CelledoniMcLachlanOwrenQuispel2013}
is the following:

\begin{theorem}[\cite{CelledoniMcLachlanOwrenQuispel2013}]
    Consider a Hamiltonian vector field \eqref{eq:firstordham}
    with cubic Hamiltonian i.e. the function $H$ is a degree 3 polynomial.
    Then its KHK discretisation \eqref{eq:kahan2} is birational
    and admits the following invariant:
    \begin{equation} 
        \widetilde{H}\left( \vec{x},h \right)=
        H\left( \vec{x} \right) + \frac{h}{3}\grad H\left( \vec{x} \right)^{T}
        \left( I_{N}-\frac{h}{2}\vec{f'} \left( \vec{x} \right) \right)^{-1}\vec{f}\left( \vec{x} \right).
        \label{eq:Htilde}
    \end{equation}
    Moreover the KHK discretisation \eqref{eq:kahan} preserves the
    following measure on $\R^{n}$:
    \begin{equation}
        m\left( \vec{x} \right)
        =
        \frac{\ud x_{1} \wedge \dots \ud x_{N}}{\det\left( I_{N}-h\vec{f'}\left( \vec{x} \right)/2 \right)}.
        \label{eq:measure}
    \end{equation}
    \label{thm:invariant}
\end{theorem}

\begin{remark}
    The invariant \eqref{eq:Htilde} is the ratio of two polynomials.  If
    $N$ is even (resp.\ odd), the degree of the numerator is at most $N+1$
    (resp. $N+2$), and the degree of the denominator is at most $N$ (resp.\
    $N-1$). In the particular case $N=2$, which we will consider in this
    paper, the invariant \eqref{eq:Htilde} has the following shape:
    \begin{equation} \widetilde{H}\left( x,y,h \right)= \frac{C\left( x,y,h
    \right)}{D\left( x,y,h \right)}, \label{eq:Htilde2} \end{equation}
    where $\deg C = 3$ and $\deg D =2$.
    \label{rem:degre}
\end{remark}

While \Cref{thm:invariant} does not requires integrability in this paper we
restrict to the two-dimensional case, because we are interested in
integrable discretisation.  To be more specific, we show that in the
two-dimensional case the integrability of the KHK discretisation of a cubic
Hamiltonian system is completely characterised by its singularities.
Moreover, we show that these singularities lie on the vertices of a
``hexagon'' and the invariant can be written as the product of the ratios
of affine polynomials defining the prolongation of the three parallel sides
of a hexagon. More importantly, these lines are the singular fibres of the
pencil associated with the invariant.  Our result is based on a previous
investigation of the geometry of the two-dimensional integrable KHK
discretisation given in \cite{PSS2019}.  Furthermore, our main result
permits us to write down the KHK invariant knowing only the base points,
plus trivial operations. In this sense with our result we show how to do a
\emph{KHK discretisation for dummies}.

The structure of the paper is as follows: in \textsc{Section
\ref{sec:main}} we give the preliminary definitions we will use throughout
the paper and prove our main result: Theorem \ref{thm:triangles}.
\textsc{Section \ref{sec:ex}} is devoted to examples of the general
construction. We also present an example belonging to a different class of
integrable KHK discretisations presented in
\cite{CelledoniMcLachlanMcLarenOwrenQuispel2017}, but lying outside Theorem
\ref{thm:triangles}. In such a case, we show that a similar result holds,
even though the invariant is not the product of ratios of parallel lines.
In \textsc{Section \ref{sec:conclusions}} we summarise our results and
discuss open questions, motivated both by the general results and the
considered examples. In particular, we discuss that some results that
recently appeared in the literature, see
\cite{Alonso_et_al2022,GG_cremona3}, hint at a possible extension of this
work to higher-dimension, at least in the integrable case.

\section{Main result}
\label{sec:main}

In this section we state the preliminary definitions and then proceed
to state and prove the main result of this paper contained in Theorem
\ref{thm:triangles}

\subsection{Preliminaries}

Consider a pencil of curves in the affine plane $\Cp^{2}$: 
\begin{equation}
    p\left( x,y;e_{0}:e_{1} \right) = e_{0} p_{0}\left( x,y \right) 
    + e_{1} p_{1}\left( x,y \right),
    \quad
    \left[ e_{0}:e_{1} \right]\in\Pj^{1}.
    \label{eq:pencil}
\end{equation}
Then, we recall the following definitions:

\begin{definition}
    Given a pencil of plane curves $p\left( x,y;e_{0}:e_{1} \right)$
    if a point $\left(x_{0},y_{0} \right)\in\Cp^{2}$ is such that
    $p_{0}\left( x_{0},y_{0} \right)=p_{1}\left( x_{0},y_{0} \right)=0$,
    then it is called a \emph{base point} of the pencil \eqref{eq:pencil}.
    \label{def:basepoints}
\end{definition}

\begin{definition}
    Given a pencil of plane curves $p\left( x,y;e_{0}:e_{1} \right)$
    if a point $\left(x_{0},y_{0};e_{0}',e_{1}'  \right)\in\Cp^{2}\times\Pj^{1}$ 
    is such that
    \begin{equation}
        p\left( x_{0},y_{0};e_{0}':e_{1}' \right) = 
        \frac{\partial p}{\partial x}\left( x_{0},y_{0};e_{0}':e_{1}' \right)=
        \frac{\partial p}{\partial y}\left( x_{0},y_{0};e_{0}':e_{1}' \right)=0.
        \label{eq:singdef}
    \end{equation}
    then it is called a \emph{singular point} for the pencil \eqref{eq:pencil}.
    \label{def:singular}
\end{definition}
%\begin{definition}
%    Then:
%    \begin{enumerate}
%        \item a point $\left(x_{0},y_{0} \right)\in\Cp^{2}$ such that
%            $h_{0}\left( x_{0},y_{0} \right)=h_{1}\left( x_{0},y_{0} \right)=0$
%            is called a \emph{base point} of the pencil \eqref{eq:pencil}.
%        \item a point $\left(x_{0},y_{0};e_{0}',e_{1}'  \right)\in\Cp^{2}\times\Pj^{1}$ 
%            such that
%            \begin{equation}
%                p\left( x_{0},y_{0};e_{0}',e_{1}' \right) = 
%                \frac{\partial p}{\partial x}\left( x_{0},y_{0};e_{0}',e_{1}' \right)=
%                \frac{\partial p}{\partial y}\left( x_{0},y_{0}';e_{0},e_{1}' \right)=0.
%                \label{eq:singdef}
%            \end{equation}
%            is called a \emph{singular point} for the pencil \eqref{eq:pencil}.
%    \end{enumerate}
%    \label{def:singular}
%\end{definition}

Intuitively, a base point is a point lying on \emph{each curve} of the
pencil \eqref{eq:pencil}. On the other hand, a singular point lies on the
curve \emph{and} its gradient vanishes. This means that, in general, for
cubic pencils the singular points lie only on specific members of a pencil,
called \emph{singular fibres}. More formally:
\begin{definition}
    Given a pencil of plane curves $p\left( x,y;e_{0}:e_{1} \right)$
    if the curve $p_{s}(x,y):= p\left( x,y;e_{0}':e_{1}' \right)$,
    with $\left[e_{0}':e_{1}'  \right]\in\Pj^{1}$ contains a singular
    point then it is called a \emph{singular fibre} of the pencil
    \eqref{eq:pencil}.  \label{def:singularfibre}
\end{definition}

If the pencil $p$ is a pencil of elliptic curves, on a
singular fibre either the \emph{genus drops to zero} or the polynomial is
\emph{factorisable}.  A general classification of the singular fibres
of elliptic curves is due to Kodaira~\cite{Kodaira1963}. In addition,
all the possible arrangements of singular fibres on an elliptic fibration
have been classified in~\cite{OS1991}.  This classification is reported in
the monograph~\cite{SchuttShioda2019mordell}, where the different elliptic
fibrations are distinguished using the associated Dynkin diagram, of the
$A$, $D$, $E$ series, see~\cite[Proposition 5.15]{SchuttShioda2019mordell}
The application of this theory to discrete integrable systems has been
discussed in the monograph~\cite{Duistermaat2011book}, and more recently
in~\cite{GubShiNahm}.

In the literature on the algebro-geometric structure of integrable
systems the notion of singular fibres has appeared in several cases.
For instance, in \cite{PettigrewRoberts2008} a classification of the
singular fibres of the QRT biquadratics, (see \cite{QRT1988,QRT1989})
was presented. In \cite{CarsteaTakenawa2012} it was noted that for
minimal elliptic curves of degree higher than three the singular fibre
is unique. Finally, in \cite{Carsteaetal2017} the notion of singular
fibre was used to build several de-autonomisations of QRT maps, see
\cite{HietarintaJoshiNijhoff2016}.

Consider now a birational map $\boldsymbol{\Phi}\colon\Cp^{2}\to\Cp^{2}$.
As usual an \emph{invariant} is a scalar function $h=h(x,y)$ constant 
under iteration of the birational map $h(\boldsymbol{\Phi}(x,y))=h(x,y)$.  In
the case of a rational invariant $h=p_{0}/p_{1}$, with $p_{i}\in\Cp[x,y]$,
the associated pencil $p=e_{0} p_{0} + e_{1} p_{1}$ is \emph{covariant}
with respect to the map $\boldsymbol{\Phi}$. So, in general we have
a one-to-one correspondence between covariant pencils of curves and
rational invariants, and we will go from one to the other
throughout the paper.

Birational maps are not always defined on $\Cp^{2}$. Using projective
geometry it is possible to give a meaning to the cases when a denominator
goes to zero, but there are still undetermined points, defined as follows:

\begin{definition}
    Consider a birational map $\boldsymbol{\Phi}\colon\Cp^{2}\to\Cp^{2}$. 
    A point $\left( x_{0},y_{0} \right)\in\Cp^{2}$ such that all the
    entries of $\boldsymbol{\Phi}$ or its inverse $\boldsymbol{\Phi}^{-1}$
    is of the form $0/0$ is called an \emph{indeterminacy point}.
    \label{def:singular2}
\end{definition}

In the integrable case the set of indeterminacy points of the map
and the set of base points of the associated covariant pencil are
the same, see \cite{CarsteaTakenawa2012,Duistermaat2011book,Tsuda2004}. In
the non-in\-te\-gra\-ble case the analysis of the singularities
proves the non-in\-te\-gra\-bi\-li\-ty of the birational map, see
\cite{DillerFavre2001}.
In particular, for non-integrable systems the analysis of singularities
can prove that the algebraic entropy of the system is positive,
(meaning that the system is non-integrable \cite{BellonViallet1999}),
and that no invariant exists \cite{Takenawa2001JPhyA}.

\subsection{Main theorem and its proof}

We state and prove the following result:
\begin{theorem}
    Consider a cubic Hamiltonian $H=H\left( x,y \right)$.
    Then, the invariant \eqref{eq:Htilde2} is representable as the ratio of 
    two products of three affine polynomials:
    \begin{equation}
        \tilde{H}\left( x,y \right) =
        \frac{\ell\left(x,y; \mu_{1},b_{2} \right)\ell\left(x,y; \mu_{2},b_{6} \right)\ell\left(x,y; \mu_{3},b_{4} \right)}{%
        \ell\left(x,y; \mu_{1},b_{5} \right)\ell\left(x,y; \mu_{2},b_{3} \right)\ell\left(x,y; \mu_{3},b_{1} \right)}
        \label{eq:Htildefact}
    \end{equation}
    where:
    \begin{equation}
        \ell \left(x,y; \mu,b \right)= y-\mu x -b.
        \label{eq:linpol}
    \end{equation}
    \label{thm:triangles}
\end{theorem}

\begin{remark}
    The lines in \eqref{eq:Htildefact} are three pairs of parallel
    lines:
    \begin{subequations}
        \begin{gather}
        \ell\left(x,y; \mu_{1},b_{2} \right)\parallel\ell\left(x,y; \mu_{1},b_{5} \right),
        \\
        \ell\left(x,y; \mu_{2},b_{3} \right)\parallel\ell\left(x,y; \mu_{2},b_{6} \right),
        \\
        \ell\left(x,y; \mu_{3},b_{1} \right)\parallel\ell\left(x,y; \mu_{3},b_{4} \right).
        \end{gather}
        \label{eq:parallel}
    \end{subequations}
    These lines intersect pairwise in the following six points in 
    the finite part of the plane $\Cp^{2}$:
    \begin{equation}
        \begin{aligned}
            B_{1} =
            \left({\frac {{  b_1}-{  b_6}}{\mu_2-\mu_3}},
            {\frac {{  b_1}\,\mu_2-\mu_3\,{  b_6}}{\mu_2-\mu_3}}\right),
            &\quad
            B_{2} =
            \left( {\frac {{  b_1}-{  b_2}}{\mu_1-\mu_3}},
            {\frac {{  b_1}\,\mu_1-\mu_3\,{  b_2}}{\mu_1-\mu_3}}\right),
            \\
            B_{3} =
            \left(-{\frac {{  b_2}-{  b_3}}{\mu_1-\mu_2}},
            -{\frac {{  b_2}\,\mu_2-\mu_1\,{  b_3}}{\mu_1-\mu_2}}\right),
            &\quad
            B_{4} =
            \left(-{\frac {{  b_3}-{  b_4}}{\mu_2-\mu_3}},
            -{\frac {{  b_3}\,\mu_3-\mu_2\,{  b_4}}{\mu_2-\mu_3}}\right),
            \\
            B_{5} =
            \left({\frac {{  b_4}-{  b_5}}{\mu_1-\mu_3}},
            {\frac {{  b_4}\,\mu_1-\mu_3\,{  b_5}}{\mu_1-\mu_3}}\right),
            &\quad
            B_{6} =
            \left(-{\frac {{  b_5}-{  b_6}}{\mu_1-\mu_2}},
            -{\frac {{  b_5}\,\mu_2-\mu_1\,{  b_6}}{\mu_1-\mu_2}}\right).
        \end{aligned}
        \label{eq:bps}
    \end{equation}
    %and the following points at infinity:
    In general, any combinations of three or more points of the
    previous list are not collinear. 
    A set of points with such property is said to be \emph{in general position}.
    %When those points are collinear the corresponding case is particular.
    %We will discuss these particular cases later. 
    \label{rem:linesandpb}
\end{remark}

The proof of theorem \ref{thm:triangles} is based on the following technical lemmas:
\begin{lemma}[\cite{PSS2019}]
    Consider a cubic Hamiltonian $H=H\left( x,y \right)$.
    Then, the invariant \eqref{eq:Htilde2} is represented by the ratio
    of the following polynomials:
    %of the KHK discretisation of a degree 3 Hamiltonian vector field
    %in the variables $\left( x,y \right)$ 
    \begin{subequations}
        \begin{align}
            C\left( x,y,h \right) &=
            (y-\mu_1x) (y-\mu_2 x) (y-\mu_3 x)+c_5 x^2+c_6 x y+c_7 y^2+c_8 x+c_9 y,
            \label{eq:Cexpl}
            \\
            D\left( x,y,h \right) &=d_1 x^2+d_2 x y+d_3 y^2+d_4 x+d_5 y+ b_1 b_3 b_5-b_2 b_4 b_6.
            \label{eq:Dexpl}
        \end{align}
        \label{eq:CDexpl}%
    \end{subequations}
    The explicit form of the coefficients $c_{i}$ and $d_{i}$ is given in 
    equation \eqref{eq:params} in \Cref{app:explicit}.
    \label{lem:khkform}
\end{lemma}

\begin{remark}
    The free parameters in \eqref{eq:CDexpl} and \eqref{eq:params} depend
    on the original parameters of the cubic Hamiltonian $H$ using the
    formulas contained in \Cref{app:explicit} of \cite{PSS2019}.  All
    the parameters depend explicitly on the time-step $h$. However, we
    note that to give the proof of \Cref{thm:triangles} we don't need to
    use this explicit expression, but it is sufficient that given the
    polynomials \eqref{eq:CDexpl} it is possible to uniquely determine the
    corresponding KHK discretisation.  This in turn implies that the result
    of Theorem \ref{thm:triangles} holds independently of the KHK structure
    of the underlying continuous system.  
    \label{rem:parameters}
\end{remark}

\begin{remark}
    Since $\deg D =2$ \eqref{eq:Dexpl} it follows that the base points
    of a KHK map lie on a conic section, i.e. on a curve of genus
    zero~\cite{Shafarevich1994}. Moreover, by explicit computation, the
    set $\mathcal{D}=\left\{ D=0 \right\}$ is the common denominator of
    the maps $\boldsymbol{\Phi}$ and $\boldsymbol{\Phi}^{-1}$.  From the
    explicit form of the parameters $d_{i}$ from \Cref{app:explicit}
    it can be proved that the real part of this conic curve is either an
    ellipse or an hyperbola, but not a parabola. In Section \ref{sec:ex}
    we will see examples of base points lying both on (real) ellipses
    and hyperbolas. Finally, we observe that the fact that $\deg D=2$
    implies that when adding the line at infinity $\Pj^{2}=\Cp^{2}\cup
    \Set{t=0}$, there are always \emph{three base points at infinity}
    coming from the solutions of: 
    \begin{equation}
        C^{h}(x:y:0,h) = 0, \quad C^{h}(x:y:t,h) = t^{3} C\left(
        \frac{x}{t},\frac{y}{t},h \right).  \label{eq:homc}
    \end{equation} \label{rem:conics}
\end{remark} 

\begin{lemma}
    The pencil of cubic curves:
    \begin{equation}
        p\left( x,y;e_{0}:e_{1} \right) = e_{0} C\left( x,y,h \right) + e_{1}D\left( x,y,h \right),
        \label{eq:pencilCD}
    \end{equation}
    where the functions $C$ and $D$ are given by equation \eqref{eq:CDexpl},
    admits two singular fibres for the following values of $\left[ e_{0}:e_{1} \right]\in\Pj^{1}$:
    \begin{equation}
        \left[e_0' :e_{1}'\right] 
        =\left[ \Delta:b_2 b_4 b_6\right],
        \quad
        \left[e_0'' :e_{1}''\right] 
        =\left[ \Delta:b_1 b_3 b_5\right],
        \label{eq:e0e1singular}
    \end{equation}
    where $\Delta$ is given by equation \eqref{eq:Deltadef}.
    Moreover, the corresponding singular curves in the pencil 
    \eqref{eq:pencilCD} factorise into affine polynomials as follows:
    \begin{subequations}
        \begin{align}
            p\left(x,y;e_0' :e_{1}'\right)
            &=\Delta%\left(b_2 b_4 b_6-b_1 b_3 b_5\right)
            \ell\left(x,y; \mu_{1},b_{2} \right)
            \ell\left(x,y; \mu_{2},b_{6} \right)\ell\left(x,y; \mu_{3},b_{4} \right),
            \\
            p\left(x,y;e_0'' :e_{1}''\right)
            &=\Delta%\left( b_2 b_4 b_6-b_1 b_3 b_5\right)
            \ell\left(x,y; \mu_{1},b_{5} \right)
            \ell\left(x,y; \mu_{2},b_{3} \right)\ell\left(x,y; \mu_{3},b_{1} \right).
        \end{align}
        \label{eq:pCDsingular}
    \end{subequations}
    \label{lem:technical}
\end{lemma}

\begin{proof}
    The proof is achieved by direct computation using computer
    algebra software, e.g. \texttt{Maple}.
    In principle, we have to solve the system \eqref{eq:singdef}
    where $p$ is given by the pencil \eqref{eq:pencilCD}.
    This approach is quite cumbersome from the computational
    point of view, as it involves the solution of nonlinear algebraic
    equations.
    We propose the following approach, which proved to be easier to
    implement.
    Take a general affine polynomial with unspecified coefficients:
    \begin{equation}
        L = \alpha x + \beta y + \gamma.
        \label{eq:linegen}
    \end{equation}
    Using polynomial long division with respect to $x$ we can write:
    \begin{equation}
        p \left( x,y;e_{0}:e_{1} \right)= Q\left( x,y;e_{0}:e_{1} \right)L\left( x,y \right)+
        R\left( y;e_{0}:e_{1} \right).
        \label{eq:pLdivision}
    \end{equation}
    If we impose $R\equiv0$, then we will have $R\mid p$.
    We can obtain such conditions by setting to zero all the
    coefficients with respect to the various powers of $y$ in $R$.
    For instance the coefficient of $y^{3}$ is:
    \begin{equation}
        \Delta e_0(\beta\mu_3+\alpha)(\beta\mu_2+\alpha)(\beta\mu_1+\alpha)=0.
        \label{eq:Adef}
    \end{equation}
    So, we can choose three different values for $\alpha$.
    This already suggests that there will be three different affine
    factors.
    
    We start by choosing $\alpha=-\beta\mu_{1}$. 
    Substituting it back into $R$ we obtain the following value for
    $\left[ e_{0}:e_{1} \right]$:
    \begin{equation}
        \left[ e_{0}:e_{1} \right]
        =
        \left[ 
            \beta \left( b_{2}-b_{5} \right)\Delta:
            \beta\left( b_1 b_2 b_3 b_5-  b_2 b_4 b_5 b_6\right)- \gamma\Delta
        \right].
        \label{eq:e0e1sol0}
    \end{equation}
    This finally yields the following values for $\gamma$:
    \begin{equation}
        \gamma = -\beta b_2, - \beta b_5 , 
        -\frac{\mu_1 \beta (b_1 b_3-b_4 b_6)}{\left(b_1-b_{4}\right)\mu_{2} +\left(b_3-b_6\right) \mu_3}.
        \label{eq:Csol0}
    \end{equation}
    Substituting~\eqref{eq:Csol0} back into \eqref{eq:e0e1sol0} we obtain
    the two solutions presented in \eqref{eq:e0e1singular},
    plus a third one:
    \begin{equation}
        \left[ e_{0}''':e_{1}''' \right]=
        \left[ 
            \begin{gathered}
            \Delta (b_2-b_5)\left(\left(b_1 -b_{4}\right) \mu_2 +\left(b_3 -b_6\right) \mu_3\right):
            \\
            \left( b_1 b_3-b_4 b_6 \right)
            \left(-\Delta \mu_1+b_2 b_5 (b_1-b_4) \mu_2+b_2 b_5 (b_3-b_6) \mu_3
            \right)
            \end{gathered}
        \right].
        \label{eq:e0e1rd}
    \end{equation}
    While substituting~\eqref{eq:e0e1singular} into the pencil~\eqref{eq:pencilCD}
    we obtain the two singular fibres~\eqref{eq:pCDsingular}, 
    this third one does not give rise to a singular fibre.

    Repeating the same argument with the other possible values of
    $\alpha$ in~\eqref{eq:Adef} we obtain the same result.
    This concludes the proof.
\end{proof}

\begin{proof}[Proof of Theorem \ref{thm:triangles}] 
    Consider the pencil built with the two polynomials
    in \eqref{eq:pCDsingular}:
    \begin{equation}
        P = \varepsilon_{0} p\left(x,y;e_0' :e_{1}'\right)
        +\varepsilon_{1} p\left(x,y;e_0'' :e_{1}''\right),
        \label{eq:Peps01}
    \end{equation}
    where $\left[ \varepsilon_{0}:\varepsilon_{1} \right]\in\Pj^{1}$.
    The following invertible change of parameters:
    \begin{equation}
        \left[ e_{0}:e_{1} \right]=
        \left[ -\left( \varepsilon_{0}+\varepsilon_{1} \right)\Delta:
        b_{1}b_{3}b_{5}\varepsilon_{1}+b_{2}b_{4}b_{6}\varepsilon_{0}\right],
        \label{eq:changevars}
    \end{equation}
    transforms the pencil \eqref{eq:pencilCD} into the pencil 
    \eqref{eq:Peps01}.
    Using the result of Lemma \ref{lem:khkform} we have that the pencil
    \eqref{eq:Peps01} is covariant on the KHK discretisation of a cubic 
    Hamiltonian vector field in the variables $\left( x,y \right)$ 
    \eqref{eq:Htilde2}.
    This in turn implies that the ratio \eqref{eq:Htildefact} is an invariant
    for the KHK discretisation of a cubic Hamiltonian vector field. 
    This concludes the proof of the theorem.
\end{proof}

\begin{remark}
    An alternative proof of Theorem \ref{thm:triangles} can be obtained
    through the theory of Darboux polynomials 
    \cite{CelledoniEvripidouMcLareOwewnQuispelTapleyvanderKamp2019}.
    Indeed, consider the KHK discretisation associated with the most general 
    cubic Hamiltonian in $\left( x,y \right)$:
    %For $h=1$ (the general case will follow scaling this one properly)
    %this KHK discretisation can be written as:
    \begin{subequations}
        \begin{align}
            \frac{x'-x}{h}  & = a_2 x' x+a_3(x'y+x y')+a_4 y'y+a_6(x'+x)+a_7(y'+y)+a_9, 
            \label{eq:h3e1}
            \\
            \frac{y'-y}{h} & = -a_1x'x-a_2(x'y+x y')-a_3y'y-a_5( x'+x)-a_6( y' +y)-a_8.
            \label{eq:h3e2}
        \end{align}
        \label{eq:h3eq}
    \end{subequations}
    The coefficients $a_{i}$ are linked to the coefficients $b_{i}$ and
    $\mu_{i}$ through the results of \cite{PSS2019}, which we report in
    formula \eqref{eq:abmu} presented in \Cref{app:map}.
    Now, consider the two polynomials given in \Cref{eq:pCDsingular}
    evaluated on $\left( x',y' \right)$ from \Cref{eq:h3e1,eq:h3e2}:
    \begin{subequations}
        \begin{align}
            p\left(x',y';e_0' :e_{1}'\right)&=
            -\frac{b_{14}b_{25}b_{36}P_{1}P_{2}P_{3}}{\mu_{12}\mu_{13}\mu_{23}Q^{3}}
            p\left(x,y;e_0' :e_{1}'\right),
            \\
            p\left(x',y';e_0'' :e_{1}''\right)&=
            -\frac{b_{14}b_{25}b_{36}P_{1}P_{2}P_{3}}{\mu_{12}\mu_{13}\mu_{23}Q^{3}}
            p\left(x,y;e_0'' :e_{1}''\right).
        \end{align}
        \label{eq:pprime}%
    \end{subequations}
    where $b_{ij}=b_i-b_j$, $\mu_{ij}=\mu_i - \mu_j$ and the polynomials $P_{i}$
    and $Q$ are given in formula \eqref{eq:PQ} presented in \Cref{app:PQ}.
    This implies that the polynomials \eqref{eq:pCDsingular} are Darboux
    polynomials with the same cofactor.
    From the general theory of Darboux polynomials this implies that their
    ratio is an invariant.
    \label{rem:darboux}
\end{remark}

In Figure \ref{fig:pencil} we show an example of pencil \eqref{eq:pencilCD}
where we highlight the two singular curves $C$ and $D$.

\begin{figure}[htb]
    \centering
    \includegraphics{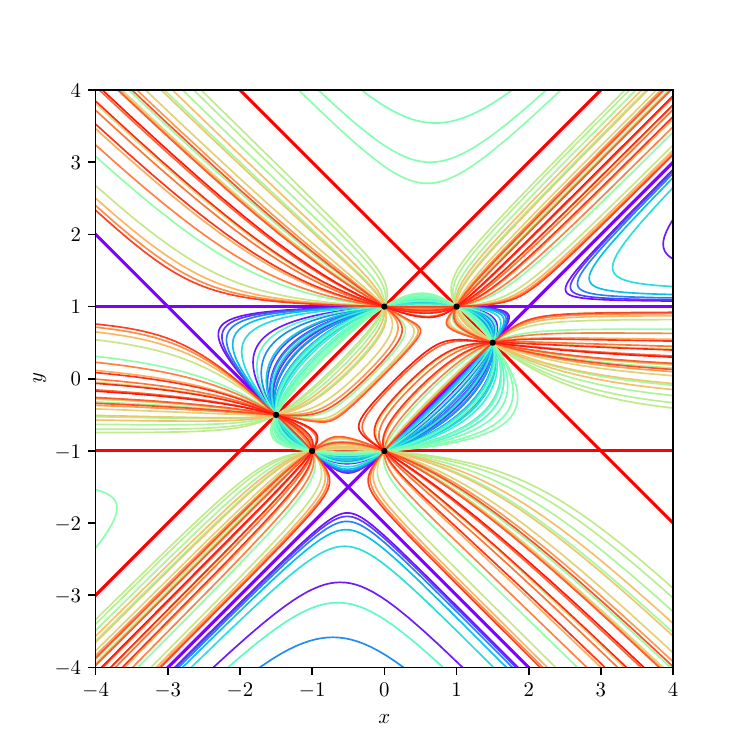}
    \caption{An example of pencil \eqref{eq:pencilCD} with 
    $b_{1}=-b_{4}=2$, $b_{2}=b_{3}=-b_{5}=-b_{6}=1$
    $\mu_{1}=0$,$\mu_{1}=1$, and $\mu_{2}=-1$.
    In red and purple are shown the two singular curves factorised
    in three lines. The base points are highlighted in black.}
    \label{fig:pencil}
\end{figure}

Theorem \ref{thm:triangles} implies a simple algorithm to construct
the invariant of the given KHK discretisation \eqref{eq:kahan2} of a cubic 
Hamiltonian vector field \eqref{eq:firstordham}.
Using the correspondence between base points of a pencil and the indeterminacy
points of the corresponding map we obtain that given such a map the
corresponding indeterminacy points will lie on the vertices of a hexagon. 
Considering the lines obtained prolonging the edges of the hexagon we
form the lines composing the invariant \eqref{eq:Htildefact}.
In the next section we will see several examples of this phenomenon.

Before moving to the example section, we give an interpretation of the 
content of \Cref{thm:triangles} in the context of Oguiso and Shioda's 
classification of 74 types of singular fibre configurations of rational 
elliptic surfaces \cite{OS1991}:

\begin{corollary}
    Consider a cubic Hamiltonian $H=H\left( x,y \right)$ for \emph{generic
    values of the parameters}. Then, the singular fibres configurations of
    the pencil of elliptic curves associated to the invariant
    \eqref{eq:Htilde2} are of type $A_{2}^{2} \oplus A_{1}$.
    \label{cor:singfib}
\end{corollary}

\begin{proof}
    From \Cref{lem:khkform} and the proof of \Cref{thm:triangles} we know 
    that the pencil of elliptic curves associated to the invariant 
    \eqref{eq:Htilde2} has two different representations, one given by 
    \Cref{eq:pencilCD} and one given by \Cref{eq:Peps01}.

    From \Cref{rem:conics} we have that the fibre $[e_{0}:e_{1}]=[0:1]$ is 
    singular. Compactifying again to $\Pj^{2}$ we have that the zero locus
    $\mathcal{D}$ is reducible:
    \begin{equation}
        \mathcal{D} = \set{tD^{h}(x:y:t,h)=0} = \Set{t=0}\cup\set{D^{h}(x:y:t,h)=0},
        \label{eq:redcon}
    \end{equation}
    where
    \begin{equation}
        D^{h}(x:y:t,h) = t^{2} D\left( \frac{x}{t},\frac{y}{t} \right).
        \label{eq:Dh}
    \end{equation}
    Now, for generic $D$, by the properties of homogenenous polynomials
    in two variables we have that:
    \begin{equation} 
        \abs{\Set{t=0}\cap\set{D^{h}(x:y:t,h)=0}}=2,
        \label{eq:cardtDh}
    \end{equation}
    i.e. the singular fibre associated to $\mathcal{D}$ is of type
    $A_{1}$ (two non-tangential intersections).

    From \Cref{eq:Peps01} we have two singular fibres at 
    $[\varepsilon_{0}:\varepsilon_{1}]=[1:0],[0:1]$. In both cases, we 
    have three lines, which for generic values of the parameters
    intersect in three different points. That is, they form
    two singular fibres of type $A_{2}$.
    This concludes the proof of the corollary.
\end{proof}

\begin{remark}
    The rational elliptic surface with singular fibres configuration of
    type $A_{2}^{2}\oplus A_{1}$ is listed in \cite[Table
    8.2]{SchuttShioda2019mordell} as number 20. Note that, for particular
    values of the parameters, cases whose singular fibre configuration
    \emph{contains} $A_{2}^{2}\oplus A_{1}$ are possible, e.g. number 40 or
    number 61.
    \label{rem:typeandpart}
\end{remark}

%\begin{corollary}
%    Consider the KHK discretisation \eqref{eq:kahan2} of a degree 3 
%    Hamiltonian vector field \eqref{eq:firstordham}.
%    Assume that the indeterminacy points of the map $\boldsymbol{\Phi}\left( x,y,h \right)$
%    given by the six points $B_{i}=\left( x_{i},y_{i} \right)\in\Cp^{2}$,
%    $i=1,2,\dots,6$ are in general position.
%    Then, the indeterminacy points are on the vertices of an hexagon
%    and the invariant of \eqref{eq:kahan2} is given by:
%    \label{cor:basepoints}
%\end{corollary}

\section{Examples}
\label{sec:ex}

In this section we show in some concrete examples how to construct the
invariant from the indeterminacy points of a given map. We will also
show that a similar result holds in the case of KHK discretisation obtained
from quadratic Hamiltonians with an affine gauge function.

\subsection{Henon-Heiles potential}

Consider the so-called Henon-Heiles (HH) potential \cite{HenonHeiles1964}:
\begin{equation}
    H = \frac{y^2+x^2}{2}+yx^2-\frac{y^3}{3}.
    \label{eq:hh}
\end{equation}
The corresponding system of Hamiltonian equations is:
\begin{equation}
    \dot{x}= x^2-y^2+y,
    \quad
    \dot{y} =-2xy-x.
    \label{eq:hheq}
\end{equation}

It is well known that in the continuous case the HH potential is
factorisable in three lines forming a triangle. These three lines govern
the behaviour of the complete HH system $H' = T+H$, where $T=T\left(
p_{x},p_{y} \right)$ is the standard kinetic energy in the conjugate
momenta of $x$ and $y$, $p_{x}$ and $p_{y}$ respectively.  For a complete
discussion on this topic we refer to~\cite{Tabor1989}.

In~\cite{CelledoniMcLachlanOwrenQuispel2013} it was shown that
the continuous triangle was preserved by the KHK discretisation
of~\eqref{eq:hheq}.
Here, following Section~\ref{sec:main}, we will show that there
exist two more sets of lines which give a factorised representation
of the invariant of the discrete systems.
We will comment on how these two invariants are pushed to infinity
in the continuum limit $h\to0$.

The KHK discretisation of equation~\eqref{eq:hheq} is:
\begin{equation}
    \frac{x'-x}{h} = x'x - y'y+\frac{y'+y}{2}, 
    \quad
    \frac{y'-y}{h} = -xy'-x'y-\frac{x+x'}{2}.
    \label{eq:hheqd}
\end{equation}
The indeterminacy points of the associated map
are the following six:
\begin{equation}
    \begin{aligned}
        B_{1} = \left(\frac{\sqrt{3}}{4}+\frac{1}{2h}, \frac{1}{4}-\frac{\sqrt{3}}{2h}  \right),
        B_{2} = \left( \frac{1}{h},-\frac{1}{2} \right),
        B_{3} = \left(-\frac{\sqrt{3}}{4}+\frac{1}{2h}, \frac{1}{4}+\frac{\sqrt{3}}{2h}\right),
        \\
        B_{4} = \left(\frac{\sqrt{3}}{4}-\frac{1}{2h}, \frac{1}{4}+\frac{\sqrt{3}}{2h}  \right),
        B_{5} = \left( -\frac{1}{h},-\frac{1}{2} \right), 
        B_{6} =\left(-\frac{\sqrt{3}}{4}-\frac{1}{2h}, \frac{1}{4}-\frac{\sqrt{3}}{2h}  \right)
    \end{aligned}
    \label{eq:HHbp}
\end{equation}
The indeterminacy points are numbered in clock-wise direction
and lie on the vertices of a regular hexagon.
Following remark \ref{rem:conics} we observe that these base points
lie on the circle:
\begin{equation}
    x^2 + y^2=\frac{1}{4}+\frac{1}{h^{2}},
    \label{eq:hhcicle}
\end{equation}
of radius $r = \sqrt{1/4+1/h^{2}}$.

Following the algorithm presented at the end of Section \ref{sec:main}
we introduce the following set of lines:
\begin{subequations}
    \begin{align}
        \overline{B_{1}B_{2}} &= 
        y-{\frac { \left( 3 h-2 \sqrt {3} \right) x}{h\sqrt {3}-2}}
        + {\frac {\sqrt {3}{h}^{2}-4 \sqrt {3}+4 h}{2h \left( h\sqrt {3}-2 \right) }},
        \label{eq:hhb1b2}
        \\
        \overline{B_{2}B_{3}} &=y+{\frac { \left( 3 h+2 \sqrt {3} \right) x}{h\sqrt {3}+2}}
        + {\frac {\sqrt {3}{h}^{2}-4 \sqrt {3}-4 h}{2h \left( h\sqrt {3}+2 \right) }},
        \label{eq:hhb2b3}
        \\
        \overline{B_{3}B_{4}} &=y-{\frac {h+2 \sqrt {3}}{4h}},
        \label{eq:hhb3b4}
        \\
        \overline{B_{4}B_{5}} &=y-{\frac { \left( 3 h+2 \sqrt {3} \right) x}{h\sqrt {3}+2}}
        +{\frac {\sqrt {3}{h}^{2}-4 \sqrt {3}-4 h}{2h \left( h\sqrt {3}+2 \right) }},
        \label{eq:hhb4b5}
        \\
        \overline{B_{5}B_{6}} &=y+{\frac { \left( 3 h-2 \sqrt {3} \right) x}{h\sqrt {3}-2}}
        +{\frac {\sqrt {3}{h}^{2}-4 \sqrt {3}+4 h}{2h \left( h\sqrt {3}-2 \right) }},
        \label{eq:hhb5b6}
        \\
        \overline{B_{6}B_{1}} &=y +{\frac {h-2 \sqrt {3}}{4h}}.
        \label{eq:hhb6b1}
    \end{align}
    \label{eq:hhlines}
\end{subequations}
We then build the invariant \eqref{eq:Htildefact} as:
\begin{equation}
    \tilde{H} =
    \frac{\overline{B_1B_2}\,\overline{B_3B_4}\,\overline{B_5B_6}}{%
    \overline{B_2B_3}\, \overline{B_4B_5}\, \overline{B_6B_1}}.
    \label{eq:hhHtildefact}
\end{equation}
Taking the continuum limit $h\to0$ we have:
\begin{equation}
    \tilde{H} 
    =-1-\sqrt{3} h-\frac{3}{2}h^2+
    -\left( \frac{4}{\sqrt{3}}H+\frac{19}{36}\sqrt{3}  \right)h^{3} + O\left( h^{4} \right),
    \label{eq:hhHtildeclim}
\end{equation}
so we see that we recover the continuum first integral \eqref{eq:hh}.

Given the hexagon formed by $B_{i}$ we can construct three additional
lines:
\begin{equation}
    \overline{B_{1}B_{4}} = y - 1 +\sqrt{3}x,
    \quad
    \overline{B_{3}B_{6}} = y - 1 -\sqrt{3}x,
    \quad
    \overline{B_{2}B_{5}} = y + \frac{1}{2},
    \label{eq:hhaddline}
\end{equation}
These lines form the original triangle of the continuous potential HH system.
The we can form the polynomial:
\begin{equation}
    P = 
    \overline{B_{1}B_{4}}\,\overline{B_{3}B_{6}}\,\overline{B_{2}B_{5}},
    \label{eq:hhP}
\end{equation}
and prove by direct computation that it is the Darboux polynomial
with the same cofactor as the numerator and denominator
of \eqref{eq:hhHtildefact}, see Remark \ref{rem:darboux}.
This implies that in the potential HH case we can construct the two
additional following invariants:
\begin{equation}
    \tilde{H}_{1} =
    \frac{\overline{B_{1}B_{4}}\,\overline{B_{3}B_{6}}\,\overline{B_{2}B_{5}}}{%
    \overline{B_2B_3}\, \overline{B_4B_5}\, \overline{B_6B_1}}.
    %\frac{\overline{B_1B_2}\,\overline{B_3B_4}\,\overline{B_5B_6}}{%
    %\overline{B_2B_3}\, \overline{B_4B_5}\, \overline{B_6B_1}}.
    \quad
    \tilde{H}_{2} =
    \frac{\overline{B_{1}B_{4}}\,\overline{B_{3}B_{6}}\,\overline{B_{2}B_{5}}}{%
    \overline{B_1B_2}\,\overline{B_3B_4}\,\overline{B_5B_6}}.
    \label{eq:hhHadd}
\end{equation}
Taking the continuum limit $h\to0$ we have:
\begin{equation}
    \tilde{H}_{1} = - \tilde{H}_{2} 
    =-\frac{9\sqrt{3}}{2h^{3}}\left( H-\frac{1}{6} \right) + O\left( \frac{1}{h^{2}} \right),
    \label{eq:hhHaddclim}
\end{equation}
where we used the fact $P=-3H+1/2$.
So we see that also in this case we recover the continuum first 
integral \eqref{eq:hh} through the continuum limit.

A graphical representation of the situation is given in Figure \ref{fig:hh}.
In particular we see that the lines in \ref{eq:hhaddline} are independent
of $h$, so they are preserved by the continuum limit.
On the other hand the lines in \ref{eq:hhlines} as $h\to0$ are pushed to
infinity.
This explains why in the continuous HH system in the finite part of the plane
only the triangle defined by \eqref{eq:hh} is present.
Finally, from a direct computation we see that the singular fibres
configuration of the pencil associated to the invariant 
\eqref{eq:hhHtildefact} is of type $A_{2}^{3}\oplus A_{1}$, and there is
a singular fibres of type $A_{0}$ represented by a nodal cubic,
i.e. it is the elliptic fibration number 61 from \cite[Table 8.2]{SchuttShioda2019mordell}.
That is, the structure is more special than the generic one, described in
\Cref{cor:singfib}, and explains the additional triangle-like structure
observed.

\begin{figure}[hbt]
    \centering
    \includegraphics{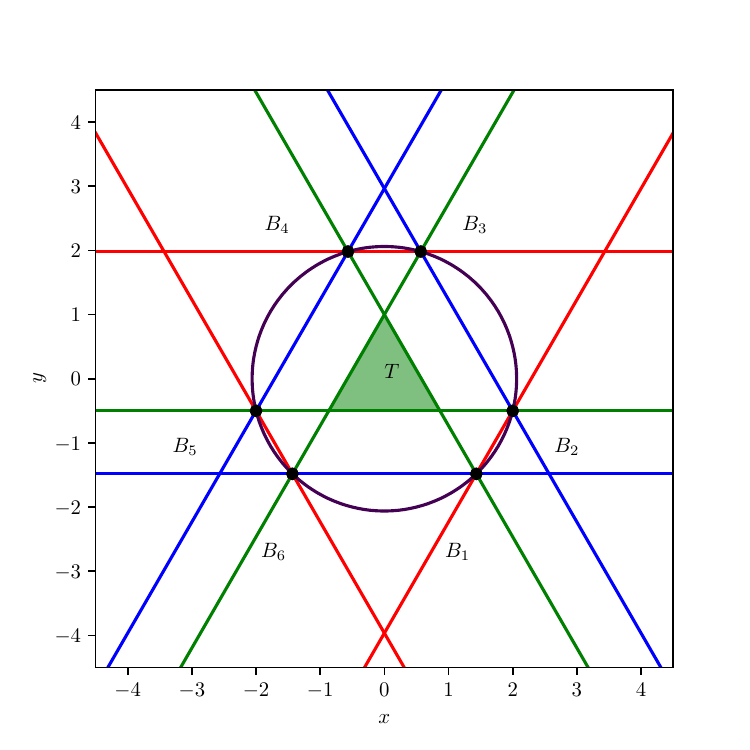}
    \caption{The HH case~\eqref{eq:hheqd} with 
        $h=1/2$: 
        the lines $\overline{B_{1}B_{2}}$, $\overline{B_{3}B_{4}}$, 
        and $\overline{B_{5}B_{6}}$ in red, the lines $\overline{B_{2}B_{3}}$, 
        $\overline{B_{4}B_{5}}$, and $\overline{B_{6}B_{1}}$ in blue, and
        the $h$ independent lines $\overline{B_{1}B_{4}}$, $\overline{B_{2}B_{5}}$, 
    and $\overline{B_{3}B_{6}}$ in green. The HH tringle~\cite{HenonHeiles1964} is visualised in light green, while the circle~\eqref{eq:hhcicle} is drawn in purple.}
    \label{fig:hh}
\end{figure}

\subsection{The most general factorisable example}

In this subsection we consider a generalisation of the HH example.
That is, we consider the most general cubic Hamiltonian $H$ factorisable
in three affine factors.
Up to canonical transformations, this Hamiltonian is:
\begin{equation}
    H = \left( x-x_{0} \right)\left( y-y_{0} \right)\left( A x + B y +C  \right),
    \label{eq:fact}
\end{equation}
where $x_{0}$, $y_{0}$, $A$, $B$, and $C$ are arbitrary constants.
The corresponding system of Hamiltonian equations is:
\begin{equation}
    \dot{x}= (x-x_0) (A x+2 B y+C-B y_0),
    \quad
    \dot{y} =-(y-y_0) (2 A x+B y+C-A x_0).
    \label{eq:facteq}
\end{equation}
The factorisation structure preserves a triangle-like configuration like
in the HH case.
We will discuss how this structure transforms after the KHK discretisation.

The KHK discretisation of equation \eqref{eq:facteq}, constructed with the
rule \eqref{eq:kahan} is
\begin{subequations}
    \begin{align}
        \frac{x'-x}{h} &
        \begin{aligned}[t]
            &= \frac{B}{2} \left[(2 y-y_0) x'-2(y- y_0+ y') x_0-x (y_0-2 y')\right]
            \\
            &-\frac{x_{0}}{2} \left[A (x+ x')+2 C\right]
            +\frac{1}{2} (2 A x+C) x'+\frac{C}{2} x , 
        \end{aligned}
        \label{eq:facteqda}
        \\
        \frac{y'-y}{h} &
        \begin{aligned}[t]
            &=- \frac{A}{2} \left[(2 x-x_0) y'-2(x- x_0+ x') y_0-y (x_0-2 x')\right]
            \\
            &+\frac{y_{0}}{2} \left[B (y+ y')+2 C\right]
            -\frac{1}{2} (2 A x+C) y'-\frac{C}{2} y, 
        \end{aligned}
        \label{eq:facteqdb}
    \end{align}
    \label{eq:facteqd}
\end{subequations}
and possesses the following indeterminacy points:
\begin{equation}
    \begin{gathered}
        B_{1} = \left(x_0,+\frac{y_0}{2} -\frac{1}{B}\left( \frac{Ax_0+C}{2}-\frac{1}{h}\right)  \right),
        \quad
        B_{2} = \left( \frac{x_0}{2}-\frac{1}{A}\left(\frac{B y_0+C}{2}-\frac{1}{h}\right), y_0 \right),
        \\
        B_{3} = \left(\frac{x_0}{2}-\frac{1}{A}\left(\frac{B y_0+C}{2}-\frac{1}{h}\right), 
        \frac{y_0}{2}-\frac{1}{B}\left(\frac{A x_0+C}{2}+\frac{1}{h}\right)\right),
        \\
        B_{4} = \left(x_0, \frac{y_0}{2}-\frac{1}{B}\left(\frac{A x_0+C}{2}+\frac{1}{h}\right)  \right),
        \quad
        B_{5} = \left( \frac{x_0}{2}-\frac{1}{A}\left(\frac{B y_0+C}{2}+\frac{1}{h}\right), y_0 \right), 
        \\
        B_{6} =\left(\frac{x_0}{2}-\frac{1}{A}\left(\frac{B y_0+C}{2}+\frac{1}{h}\right), 
        \frac{y_0}{2}-\frac{1}{B}\left(\frac{A x_0+C}{2}-\frac{1}{h}\right) \right)
    \end{gathered}
    \label{eq:factbp}
\end{equation}
The indeterminacy points are numbered in clock-wise direction
and lie on the vertices of a hexagon.
Following remark \ref{rem:conics} we observe that these base points
lie on the ellipse:
\begin{equation}
    \begin{aligned}
        \frac {{x}^{2}}{{B}^{2}}
        +\frac {{y}^{2}}{{A}^{2}} 
        +\frac {xy}{AB}
        &+\frac{1}{B^{2}} \left( \frac{C}{A}-x_0+ \right) x 
        +\frac{1}{A^{2}} \left( \frac{C}{B} -y_0\right) y
        + \left( \frac{x_{0}+y_{0}}{2} \right)^{2}
        \\
        &={\frac {1}{{h}^{2}{A}^{2}{B}^{2}}}
        +\frac{C}{2AB} \left( \frac {x0}{B} + \frac {y0}{A}\right)
        -\frac{1}{4} {\frac {{C}^{2}}{{A}^{2}{B}^{2}}}.
    \end{aligned}
    \label{eq:factell}
\end{equation}

In the same way as in the previous section we introduce the following 
set of lines:
\begin{subequations}
    \begin{align}
        \overline{B_{1}B_{2}} &= 
        x+{\frac {By}{A}}-{\frac {Ah{  x_0}+Bh{  y_0}-Ch+2}{2Ah}},
        \label{eq:factb1b2}
        \\
        \overline{B_{2}B_{3}} &=
        x-{\frac {Ah{  x_0}-Bh{  y_0}-Ch+2}{2Ah}},
        \label{eq:factb2b3}
        \\
        \overline{B_{3}B_{4}} &=
        y+{\frac {Ah x_0-Bh y_0+Ch+2}{2Bh}},
        \label{eq:factb3b4}
        \\
        \overline{B_{4}B_{5}} &=
        x+{\frac {By}{A}}-{\frac {Ah{  x_0}+Bh{  y_0}-Ch-2}{2Ah}},
        \label{eq:factb4b5}
        \\
        \overline{B_{5}B_{6}} &=
        x-{\frac {Ah{  x_0}-Bh{  y_0}-Ch-2}{2Ah}},
        \label{eq:factb5b6}
        \\
        \overline{B_{6}B_{1}} &=
        y+\frac {Ah{  x_0}-Bh{  y_0}+Ch-2}{2Bh}.
        \label{eq:factb6b1}
    \end{align}
    \label{eq:factlines}
\end{subequations}
We then build the invariant \eqref{eq:Htildefact} as:
\begin{equation}
    \tilde{H} =
    \frac{\overline{B_1B_2}\,\overline{B_3B_4}\,\overline{B_5B_6}}{%
    \overline{B_2B_3}\, \overline{B_4B_5}\, \overline{B_6B_1}}.
    \label{eq:factHtildefact}
\end{equation}
Taking the continuum limit $h\to0$ we have:
\begin{equation}
    \begin{aligned}
        \tilde{H} 
        =-1-\left( Ax_{0}+By_{0}+C \right) h &+
        \left( Ax_{0}+By_{0}+C \right)^{2}\frac{h^2}{2}+
        \\
        &+\left[2AB H\left( x,y \right) - \kappa  \right]h^{3} + O\left( h^{4} \right),
    \end{aligned}
    \label{eq:factHtildeclim}
\end{equation}
where $\kappa=\kappa\left( A,B,C,x_{0},y_{0} \right)$ is a constant.
So, also in this case the continuum first integral \eqref{eq:fact} arises
at the third order in $h$.

%Given the hexagon formed by $B_{i}$ we can construct three additional
%lines:
In addition, we have the following lines:
\begin{equation}
    \overline{B_{1}B_{4}} = x-x_{0},
    \quad
    \overline{B_{3}B_{6}} = A x + B y + C,
    \quad
    \overline{B_{2}B_{5}} = y - y_{0},
    \label{eq:factaddline}
\end{equation}
which are three factors of the original Hamiltonian \eqref{eq:fact}.
Then we can form the polynomial:
\begin{equation}
    P = 
    \overline{B_{1}B_{4}}\,\overline{B_{3}B_{6}}\,\overline{B_{2}B_{5}},
    \label{eq:factP}
\end{equation}
and prove by direct computation that it is the Darboux polynomial
with the same cofactor as the numerator and denominator
of \eqref{eq:factHtildefact}, see Remark \ref{rem:darboux}.
This implies that we can construct the two
additional following invariants:
\begin{equation}
    \tilde{H}_{1} =
    \frac{\overline{B_{1}B_{4}}\,\overline{B_{3}B_{6}}\,\overline{B_{2}B_{5}}}{%
    \overline{B_2B_3}\, \overline{B_4B_5}\, \overline{B_6B_1}}.
    %\frac{\overline{B_1B_2}\,\overline{B_3B_4}\,\overline{B_5B_6}}{%
    %\overline{B_2B_3}\, \overline{B_4B_5}\, \overline{B_6B_1}}.
    \quad
    \tilde{H}_{2} =
    \frac{\overline{B_{1}B_{4}}\,\overline{B_{3}B_{6}}\,\overline{B_{2}B_{5}}}{%
    \overline{B_1B_2}\,\overline{B_3B_4}\,\overline{B_5B_6}}.
    \label{eq:factHadd}
\end{equation}
Taking the continuum limit $h\to0$ we have:
\begin{equation}
    \tilde{H}_{1} = - \tilde{H}_{2} 
    =- A^{2}B H\left( x,y \right) h^{3} + O\left( h^{4} \right),
    \label{eq:factHaddclim}
\end{equation}
So we see that also in this case we recover the continuum first 
integral \eqref{eq:fact} through the continuum limit.

To summarise, in the most general factorisable case the three factorised
lines are preserved independently from $h$.  This explains why in the
continuous factorised system in the finite part of the plane only the
triangle defined by \eqref{eq:fact} is present.  On the other hand, the two
families of lines \eqref{eq:factlines}, alongside of the base points
\eqref{eq:factbp} are pushed to the line at infinity as $h\to0$.  See
Figure \ref{fig:fact} for a graphical representation.  Like in the case of
the HH potential, it is possible to see that the singular fibres
configuration of the pencil associated to the invariant
\eqref{eq:factHtildefact} is of type $A_{2}^{3}\oplus A_{1}$, and there is
a singular fibres of type $A_{0}$ represented by a nodal cubic, i.e. it is
the elliptic fibration number 61 from \cite[Table
8.2]{SchuttShioda2019mordell}.  So, the structure is more special than the
generic one, described in \Cref{cor:singfib}, and explains the additional
triangle-like structure observed.  To conclude, note that even though
    the most general factorisable case depends on \emph{five} parameters
after a proper choice of coordinates, can be constructed essentially in the
same way as the ``parameterless'' HH case.
 
\begin{figure}[hbt]
    \centering
    \includegraphics{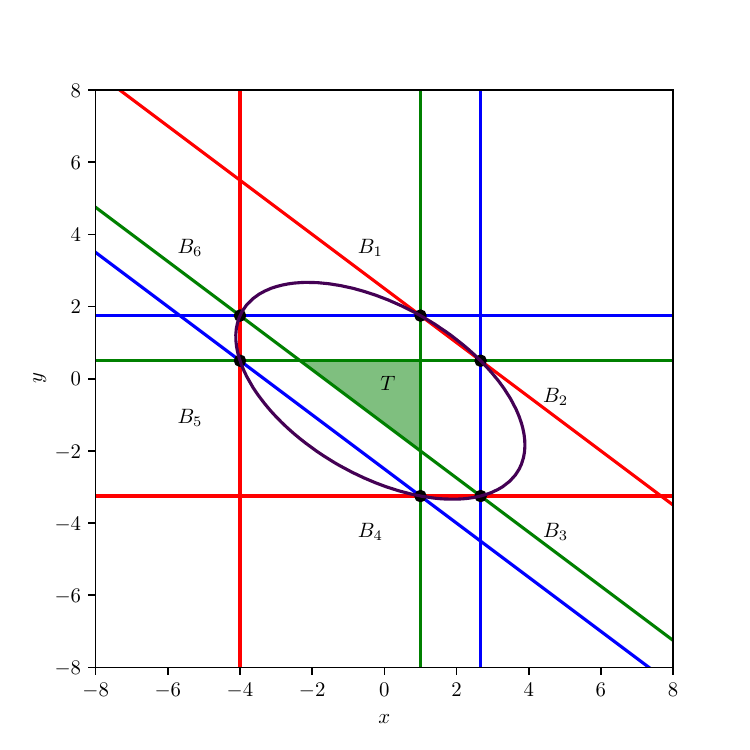}
    \caption{The general factorisable case \eqref{eq:facteqd} with $h=1/10$,
        $x_0=1$, $y_0=1/2$, $A=3$, $B=4$, and $C=5$: the lines
        $\overline{B_{1}B_{2}}$, $\overline{B_{3}B_{4}}$, and
        $\overline{B_{5}B_{6}}$ in red, the lines $\overline{B_{2}B_{3}}$,
        $\overline{B_{4}B_{5}}$, and $\overline{B_{6}B_{1}}$ in blue, and the
        $h$ independent lines $\overline{B_{1}B_{4}}$, $\overline{B_{2}B_{5}}$,
        and $\overline{B_{3}B_{6}}$ in green.  The analog of the HH triangle is
    visualised in light green, while the ellipse~\eqref{eq:factell} is drawn in
purple.}
    \label{fig:fact}
\end{figure}

\subsection{A non-factorisable example}

In the past two subsections we gave some examples of continuum Hamiltonians
factorisable in three affine polynomials. In this subsection we show what
happens in the case such factorisation is not possible. Consider the
following Hamiltonian:
\begin{equation}
    H = y\left( x^{2}-y^{2}-1  \right).
    \label{eq:nonfact}
\end{equation}
The polynomial $P=x^{2}-y^{2}-1\in\Cp\left[ x,y \right]$ is not
factorisable. So, the Hamiltonian \eqref{eq:nonfact} is made of a linear
factor and an irreducibile quadratic one.  The corresponding system of
Hamiltonian equations is:
\begin{equation}
    \dot{x}= x^2-3y^2-1,
    \quad
    \dot{y} =-2xy. 
    \label{eq:nonfacteq}
\end{equation}
In Figure \ref{fig:nonfactcont} we show the level curves
of the continuous Hamiltonian \eqref{eq:nonfact}, where it is clear
that no triple linear factorisation occurs.

\begin{figure}[hbt]
    \centering
    \includegraphics{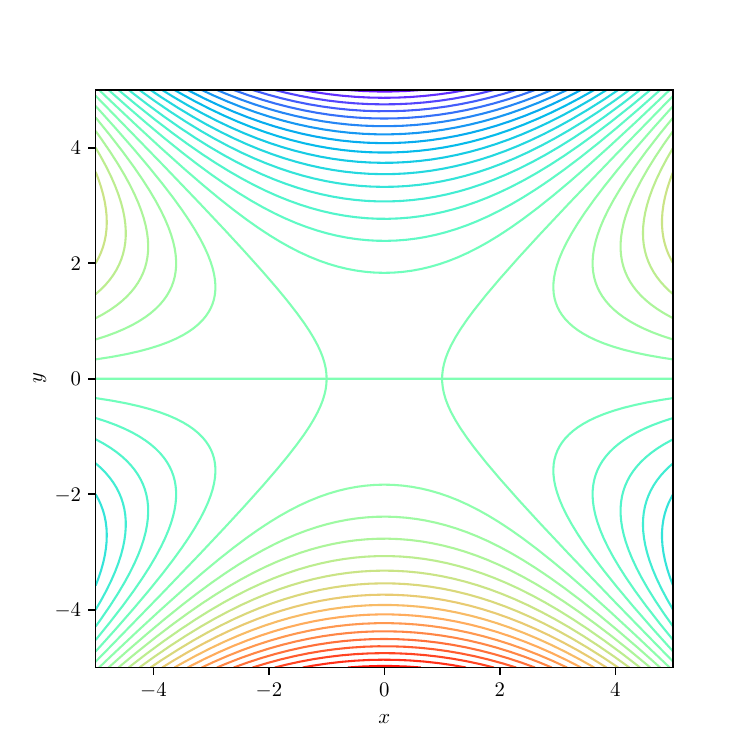}
    \caption{The level curves $H=\varepsilon$ with $H$ given 
    by equation \eqref{eq:nonfact} and 32 different values of $\varepsilon$.
    It it possible to note that there is only a linear factor (the line $y=0$) 
    and that the base points are pushed to the line at infinity in $\Pj^{2}$.}
    \label{fig:nonfactcont}
\end{figure}

Following the rule \eqref{eq:kahan} we have the following KHK discretisation
\begin{equation}
    \frac{x'-x}{h} = x x' - 3 y y'-1, 
    \quad
    \frac{y'-y}{h} = -x y'-x' y,
    \label{eq:nonfacteqd}
\end{equation}
which possesses the following indeterminacy points:
\begin{equation}
    \begin{gathered}
        B_{1} =
        \left( 
            -\frac{1}{2}\left( h+\frac{1}{h} \right),{\frac {\sqrt {\delta }}{6h}}
        \right),
        \quad
        B_{2} =
        \left( \frac{1}{2}\left( h+\frac{1}{h} \right), {\frac {\sqrt {\delta }}{6h
}}
        \right),
        \quad
        B_{3} = \left(\frac{1}{h},0 \right),
        \\
        B_{4} =
        \left( \frac{1}{2}\left( h+\frac{1}{h} \right),-{\frac {\sqrt {\delta }}{6h}}
        \right),
        \quad
        B_{5} = 
        \left( -\frac{1}{2}\left( h+\frac{1}{h} \right),-{\frac {\sqrt {\delta }}{6h}}
        \right),
        \quad
        B_{6} =\left(-\frac{1}{h},0\right),
    \end{gathered}
    \label{eq:nonfactbp}
\end{equation}
where
\begin{equation}
    \delta = 3\left( 1-h^{2} \right) \left( {h}^{2}+3 \right).
    \label{eq:deltaq}
\end{equation}
Following remark \ref{rem:conics} we observe that these base points
lie on the ellipse:
\begin{equation}
    h^2 x^2+3 h^2 y^2=1.
    \label{eq:nonfactell}
\end{equation}

\begin{remark}
    Note that $\delta>0$ if $-1 < h < 1$ which justifies taking the 
    square roots in \eqref{eq:nonfactbp}. Since we are interested in the limit
    $h\to 0^{+}$ this is no restriction. If one wishes to consider different
    values of $h$ one can consider the base points as lying on a hexagon
    on the plane in $\Pi = \R \times \imath \R \subset \Cp^{2}$.
    \label{rem:squareroot}
\end{remark}

In the same way as in the previous section we introduce the following 
set of lines:
\begin{subequations}
    \begin{align}
        \overline{B_{1}B_{2}} &= y-{\frac{\sqrt {\delta }}{6h}},
        \label{eq:nonfactb1b2}
        \\
        \overline{B_{2}B_{3}} &=
        y + \frac{1}{3} \frac{\sqrt{\delta}}{1-h^{2}} x 
        - \frac{1}{3h} \frac{\sqrt{\delta}}{1-h^{2}},
        %x+\frac{3\imath (h^2-1)}{\sqrt{\delta}}y- \frac{1}{h},
        \label{eq:nonfactb2b3}
        \\
        \overline{B_{3}B_{4}} &=
        y - \frac{1}{3} \frac{\sqrt{\delta}}{1-h^{2}} x 
        + \frac{1}{3h} \frac{\sqrt{\delta}}{1-h^{2}},
        %x-\frac{3\imath (h^2-1)}{\sqrt{\delta}}y- \frac{1}{h},
        \label{eq:nonfactb3b4}
        \\
        \overline{B_{4}B_{5}} &= y+\frac{\sqrt {\delta}}{6h},
        \label{eq:nonfactb4b5}
        \\
        \overline{B_{5}B_{6}} &=
        y + \frac{1}{3} \frac{\sqrt{\delta}}{1-h^{2}} x 
        + \frac{1}{3h} \frac{\sqrt{\delta}}{1-h^{2}},
        %x+\frac{3\imath (h^2-1)}{\sqrt{\delta}}y+ \frac{1}{h},
        \label{eq:nonfactb5b6}
        \\
        \overline{B_{6}B_{1}} &=
        y - \frac{1}{3} \frac{\sqrt{\delta}}{1-h^{2}} x 
        - \frac{1}{3h} \frac{\sqrt{\delta}}{1-h^{2}}.
        %x-\frac{3\imath (h^2-1)}{\sqrt{\delta}}y+ \frac{1}{h},
        \label{eq:nonfactb6b1}
    \end{align}%
    \label{eq:nonfactlines}
\end{subequations}
%In this case these lines are clearly complex lines lying in the plane
%$\Pi\subset\Cp^{2}$.

We then build the invariant \eqref{eq:Htildefact} as:
\begin{equation}
    \tilde{H} =
    \frac{\overline{B_1B_2}\,\overline{B_3B_4}\,\overline{B_5B_6}}{%
    \overline{B_2B_3}\, \overline{B_4B_5}\, \overline{B_6B_1}}.
    \label{eq:nonfactHtildenonfact}
\end{equation}
Taking the continuum limit $h\to0^{+}$ we have:
\begin{equation}
    \tilde{H} 
    =-1 - 4 H\left( x,y \right) h^{3} + O\left( h^{4} \right),
    \label{eq:nonfactHtildeclim}
\end{equation}
so we see that we recover the continuum first integral \eqref{eq:nonfact}.

Like in the previous cases, given the hexagon formed by $B_{i}$ we can 
construct three diagonal lines:
\begin{equation}
    \overline{B_{1}B_{4}} =
    x+\frac{3(h^2+1)}{\sqrt{\delta}}y,
    \quad
    \overline{B_{3}B_{6}} = y,
    \quad
    \overline{B_{2}B_{5}} =
    y-\frac{1}{3}\frac{\sqrt{\delta}}{h^2+1}x.
    \label{eq:nonfactaddline}
\end{equation}
Considering their product:
\begin{equation}
    P = 
    \overline{B_{1}B_{4}}\,\overline{B_{3}B_{6}}\,\overline{B_{2}B_{5}},
    \label{eq:nonfactP}
\end{equation}
we find that this polynomial is not a Darboux polynomial for the
map \eqref{eq:nonfacteqd}.
In particular we have:
\begin{equation}
    P = \left(H+y\right) + O(h),
    \label{eq:nonfactPlim}
\end{equation}
which does not reduce to the continuous Hamiltonian,
but to its factorisable part: 
$H+y=\left( x- y\right)\left( x+ y \right)$.

To prove that the only linearly factorisable singular fibres are numerator
and denominator of \eqref{eq:nonfactHtildenonfact} we consider the associated
pencil:
\begin{equation}
    p= e_{0}\overline{B_1B_2}\,\overline{B_3B_4}\,\overline{B_5B_6}+
    e_{1}\overline{B_2B_3}\, \overline{B_4B_5}\, \overline{B_6B_1}.
    \label{eq:pBi}
\end{equation}
Excluding the trivial singular fibres at $\left[ e_{0}:e_{1}
\right]=\left[ 0:1 \right]$
and $\left[ e_{0}:e_{1} \right]=\left[ 1:0 \right]$ this pencil has 
the following singular fibres:
\begin{subequations}
    \begin{align}
        p_{1,s}&= 
        y\left[ \left(1+\frac{h^{2}}{3}\right) x^2-\left(1-h^2\right) y^2-1-\frac{h^2}{3} \right]
        \label{eq:psing1}
        \\
        p_{2,s} &
        \begin{aligned}[t]
            &=\frac{\sqrt {3\delta}}{6}
            h^{3}\sqrt{3+ {h}^{2}}
            \left[ 4 {h}^{3}+\imath\sqrt{\delta}\left( 3+h^{2} \right) \right]
            \left[ 1- \left( {x}^{2}+3 {y}^{2} \right) {h}^{2} \right]
            \\
            &+\frac{{h}^{3}}{3}\left( 3+h^{2} \right)
            \left[ \imath h^{3} \sqrt{\delta}
                -\frac{\delta}{4}\left( 3+h^{2} \right)
            \right] p_{1,s}
            %\left( 
            % \left[ {x}^{2}+3 {y}^{2}-1 \right) {h}^{2}+3 {x}^{2}-3 {y}^{2}-3 \right], 
        \end{aligned}
        \label{eq:psing2}
        \\
        p_{3,s} &=
        \begin{aligned}[t]
            &\sqrt {3\delta}h^{3}
            \sqrt{3+ {h}^{2}} 
            \left[ 4 {h}^{3}- \imath \sqrt {\delta}\left( 3+h^{2} \right) \right]
            \left[1- \left( {x}^{2}+3 {y}^{2} \right) {h}^{2} \right]  
            \\
            &
            -\frac{9h^{3}}{2}\left( 3+h^{2} \right) 
            \left[ \delta\left( 3+h^{2} \right)  + 4 \imath h^{3} \sqrt{ \delta} \right]
            p_{1,s}.
            %&+3/2 {h}^{3}y \left(  \left( {x}^{2}+3 {y}^{2}-1
            % \right) {h}^{2}+3 {x}^{2}-3 {y}^{2}-3 \right)  \left( {h}^{8}+8 {h
            %}^{6}+18 {h}^{4}-4 \sqrt { \left( {h}^{2}-1 \right)  \left( {h}^{2}+
            %3 \right) ^{3}{h}^{6}}-27 \right)         
        \end{aligned}
        \label{eq:psing3}
    \end{align}%
    \label{eq:psings}%
\end{subequations}
The first singular fibre, equation \eqref{eq:psing1} is a
deformation of order $h^{2}$ of the original Hamiltonian \eqref{eq:nonfact}.
It is possible to check that the quadratic polynomial is not factorisable,
i.e. such a singular fibre is of type $A_{1}$.
In the same way, the two cubic curves in equations \eqref{eq:psing2} and
\eqref{eq:psing3} do not admit any affine factors, but rather are
nodal cubics, i.e. singular fibres of type $A_{0}$. At infinity except from 
the common $\mathcal{D}$ fibre of type $A_{1}$, see \Cref{rem:conics}, there
is no other new singular fibre.
%These two fibres are degenerate because their genus vanishes, so they are
%conics and not elliptic curves.
%\begin{subequations}
%    \begin{align}
%        p_{1,s}&= 
%        \begin{aligned}[t]
%            &-3\left( {h}^{2} +3 \right) {x}^{2}y+2 \sqrt {3}{h}^{2}{x}^{2}
%            +9  \left( h^{2}-1 \right)  {y}^{3}
%            \\
%            &-6 \sqrt {3}{h}^{2}{y}^{2}+ 3\left( {h}^{2}+3 \right) y-2 \sqrt {3}, 
%        \end{aligned}
%        \label{eq:psing1}
%        \\
%        p_{2,s} &=
%        \begin{aligned}[t]
%            &3 \left(  {h}^{2}+3\right) {x}^{2}y+2 \sqrt {3}{h}^{2}{x}^{2}
%            -9  \left( h^{2}-1 \right)  {y}^{3} 
%            \\
%            &-6 \sqrt {3}{h}^{2}{y}^{2}-3 \left( {h}^{2}+3 \right) y-2 \sqrt {3}. 
%        \end{aligned}
%        \label{eq:psing2}
%    \end{align}
%    \label{eq:psings}
%\end{subequations}
%while the pencil \eqref{eq:pBi} has genus 1.
%Moreover, we have that the genus of thes

To summarise, with this example we showed that when the continuum cubic
Hamiltonian is not factorisable the corresponding KHK discretisation
admits, in general, only two singular fibres factor in the product of three
affine polynomials. Other singular fibres, are  either union of a
line and a conic or nodal cubics. In particular this means that the
complete singular fibre configuration is of type $A_{2}^{2}\oplus
A_{1}^{2}$, i.e. number 40 from \cite[Table 8.2]{SchuttShioda2019mordell}.
These considerations underline the differences with the factorisable cases
discussed in the previous sections.  See Figure \ref{fig:nonfact} for a
graphical representation.

\begin{figure}[hbt]
    \centering
    \includegraphics{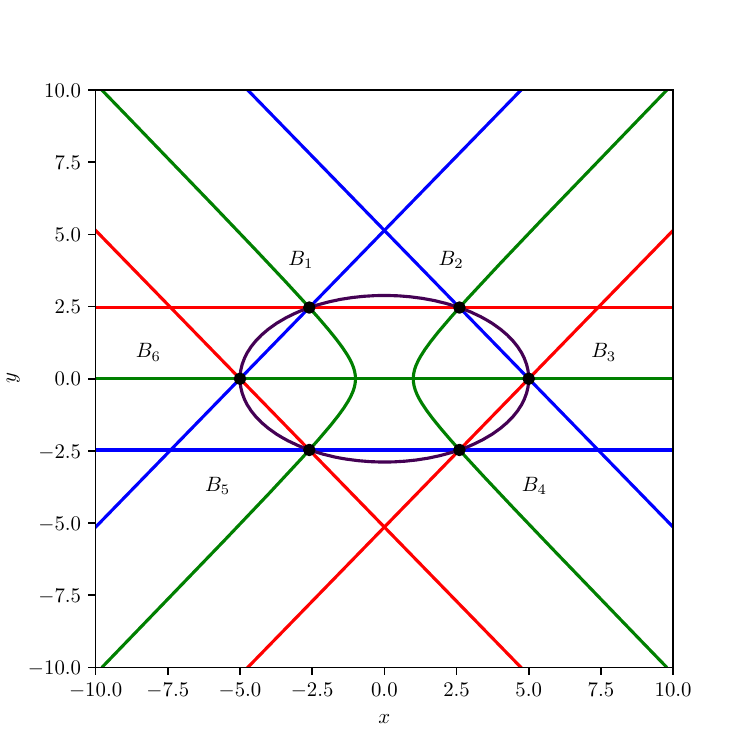}
    \caption{The non-factorisable case \eqref{eq:nonfacteqd} with 
        $h=1/5$: the lines $\overline{B_{1}B_{2}}$, $\overline{B_{3}B_{4}}$, 
        and $\overline{B_{5}B_{6}}$ in red, the lines $\overline{B_{2}B_{3}}$, 
        $\overline{B_{4}B_{5}}$, and $\overline{B_{6}B_{1}}$ in blue, and
        the singular pencil $p_{1,s}=0$ in green.
        Finally, the ellipse \eqref{eq:nonfactell} is displayed in purple.
    }
    \label{fig:nonfact}
\end{figure}

\subsection{Quadratic irreducible Hamiltonians and non-convex hexagons}

%The example we presented in the previous has a natural yet
%interesting generalisation to:
%\begin{equation}
%    H_{\alpha} = (y-\alpha)\left( x^{2}-y^{2}-1  \right).
%    \label{eq:nonfactalpha}
%\end{equation}
%For all $\alpha\in\R$ the cubic Hamiltonian is made of an quadratic
%irreducible term and a linear factor.  Depending on the value of
%$\alpha\in\R$ the base points obtained from the KHK discretisation
%of \eqref{eq:nonfactalpha} form several interesting arranVgements,
%for instance non non-convex hexagonal shapes. However, to avoid
%cumbersome formulas we restrict ourselves to the case when $\alpha=2$.

In this subsection we consider an example which shows that the base 
points can be arranged in interesting non-convex hexagonal shapes. 
The system we consider is the following:
\begin{equation}
    H_{2} = (x-2)\left( x^{2}+y^{2}-1  \right),
    \label{eq:nonfact2}
\end{equation}
which was discussed in \cite[Example 1]{CelledoniEvripidouMcLareOwewnQuispelTapleyvanderKamp2019}:
Like in the previous example the cubic Hamiltonian is made of an quadratic
irreducible term and a linear factor.
The corresponding system of Hamiltonian equations is:
\begin{equation}
    \dot{x}= 2 (x-2) y,
    \quad
    \dot{y} =- 3x^2-y^2-4 x +1.
    \label{eq:nonfacteq2}
\end{equation}
%In Figure \ref{fig:nonfactcont} we show the level curves
%of the continuous Hamiltonian \eqref{eq:nonfact}, where it clear
%that no triple linear factorisation occurs.
%
%\begin{figure}[hbt]
%    \centering
%    \includegraphics{nonfactpencil}
%    \caption{The level curves $H=\varepsilon$ with $H$ given 
%    by equation \eqref{eq:nonfact} and 32 different values of $\varepsilon$.
%    It it possible to note that there is only a linear factor (the line $y=0$) 
%    and that the base points are pushed to line at infinity in $\Pj^{2}$.}
%    \label{fig:nonfactcont}
%\end{figure}

Following the rule \eqref{eq:kahan} we have the following KHK discretisation
\begin{equation}
    \frac{x'-x}{h} = (x-2) y'+y (x'-2), 
    \quad
    \frac{y'-y}{h} = (-3 x+2) x'-y' y+2 x+1.
    \label{eq:nonfacteqd2}
\end{equation}
possessing the following indeterminacy points:
\begin{equation}
    \begin{gathered}
        B_{1} =
        \left(2,\frac{1}{h}\right),
        \quad
        B_{2} =
        \left(
        \frac {10 {h}^{3}+\sqrt {\delta_{2}}-6 h}{2h\left( 4 {h}^{2}-3 \right) },
        -\frac {2 h\sqrt {\delta_{2} }-{h}^{2}+3}{2h \left( 4 {h}^{2}-3 \right) }
        \right),
        \\
        B_{3} = 
        \left(
            \frac {10 {h}^{3}+\sqrt {\delta_{2}}-6 h}{2h \left( 4 {h}^{2}-3 \right) },
            \frac {2 h\sqrt {\delta_{2}}-{h}^{2}+3}{2h \left( 4 {h}^{2}-3 \right) }
        \right),
        \quad
        B_{4} = \left(2,-\frac{1}{h}\right),
        \\
        B_{5} = 
        \left(
            \frac {-10 {h}^{3}+\sqrt {\delta_{2}}+6 h}{2h \left( 4 {h}^{2}-3 \right) },
            \frac {2 h\sqrt {\delta_{2}}+{h}^{2}-3}{2h \left( 4 {h}^{2}-3 \right) }
        \right),
        \\
        B_{6} =
        \left(
            \frac {-10 {h}^{3}+\sqrt {\delta_{2}}+6 h}{2h \left( 4 {h}^{2}-3 \right) },
            \frac {2 h\sqrt {\delta_{2} }+{h}^{2}-3}{2h \left( 4 {h}^{2}-3 \right) }
        \right)
    \end{gathered}
    \label{eq:nonfactbp2}
\end{equation}
where
\begin{equation}
    \delta_{2} = -3  \left(1-  h^{2}\right)  \left(3- 7 {h}^{2} \right).
    \label{eq:deltaq2}
\end{equation}
Following remark \ref{rem:conics} we observe that these base points
lie on the conic:
\begin{equation}
    3 h^2 x^2-h^2 y^2-8 h^2 x+4 h^2+1=0,
    \label{eq:nonconhyp}
\end{equation}
whose real part represents an hyperbola.

\begin{remark}
    Differently from Remark \ref{rem:squareroot} $\delta_{2}$ is
    positive if only if $3/7<h^{2}<1$.
    This implies that to take the limit $h\to0$ we will go through
    a region where the base points lie in the complex space $\Cp^{2}$.
    However, since the proof of \Cref{thm:triangles} is based
    on algebraic geometry, we can still apply it.
    To draw pictures in this subsection we will assume that the
    base points lie within this range, so that they are points in the real
    plane.
    \label{rem:squareroot2}
\end{remark}

%\begin{remark}
%    Note that in this case the base points \eqref{eq:nonfactbp2} lie on
%    the following conic:
%    which is the common denominator of the maps $\boldsymbol{\Phi}$ and
%    $\boldsymbol{\Phi}^{-1}$ obtained from \eqref{eq:nonfacteqd2}.
%    This curve is an hyperbola whose zero locus becomes empty
%    as $h\to0$.
%    \label{rem:hyperbola}
%\end{remark}

In the same way as in the previous section we introduce the following 
set of lines:
\begin{subequations}
    \begin{align}
        \overline{B_{1}B_{2}} &= 
        y+
        \frac { 2 h\sqrt {\delta_{2}} +7h^{2}-3 }{6 h(1- {h}^{2})+\sqrt {\delta_{2}}} x 
        -\frac{1}{h}{\frac {4 \sqrt {\delta_{2}}{h}^{2}+8 {h}^{3}+\sqrt {\delta_{2}}}{6 h(1 -{h}^{2})+\sqrt {\delta_{2}}}},
        \label{eq:nonfactb1b22}
        \\
        \overline{B_{2}B_{3}} &=
        x-{\frac {10 {h}^{3}+\sqrt {\delta_{2}}-6 h}{2h \left( 4 {h}^{2}-3 \right) }},
        \label{eq:nonfactb2b32}
        \\
        \overline{B_{3}B_{4}} &=
        y
        -{\frac{ 2 h\sqrt {\delta_{2}}+7 {h}^{2}-3}{6h(1- {h}^{2})+\sqrt {\delta_{2}}}}x
        +\frac{1}{h}{\frac {4 \sqrt {\delta_{2}}{h}^{2}+8 {h}^{3}+\sqrt {\delta_{2}}}{6h(1 -6 {h}^{2})+\sqrt {\delta_{2}} }},
        \label{eq:nonfactb3b42}
        \\
        \overline{B_{4}B_{5}} &= 
        y
        -{\frac { 2 h\sqrt {\delta_{2}}-7 {h}^{2}+3}{6h(6 {h}^{2}-1)+\sqrt {\delta_{2}}}}x
        +\frac{1}{h}{\frac {4 \sqrt {\delta_{2}}{h}^{2}-8 {h}^{3}+\sqrt {\delta_{2}}}{ 6h(6 {h}^{2}-1)+ \sqrt {\delta_{2}}}},
        \label{eq:nonfactb4b52}
        \\
        \overline{B_{5}B_{6}} &=
        x- {\frac {10{h}^{3}-\sqrt {\delta_{2}}-6 h}{2h \left( 4 {h}^{2}-3 \right) }},
        \label{eq:nonfactb5b62}
        \\
        \overline{B_{6}B_{1}} &=
        y+
        \frac { 2 h\sqrt {\delta_{2}}-7{h}^{2}+3 }{6h( {h}^{2}-1)+\sqrt {\delta_{2}}} x
        -\frac{1}{h}
        {\frac {4 \sqrt {\delta_{2}}{h}^{2}-8 {h}^{3}+\sqrt {\delta_{2}}}{ 6 h({h}^{2}-1)+\sqrt {\delta_{2}}}}
        \label{eq:nonfactb6b12}
    \end{align}%
    \label{eq:nonfactlines2}
\end{subequations}

We then build the invariant \eqref{eq:Htildefact} as:
\begin{equation}
    \tilde{H} =
    \frac{\overline{B_1B_2}\,\overline{B_3B_4}\,\overline{B_5B_6}}{%
    \overline{B_2B_3}\, \overline{B_4B_5}\, \overline{B_6B_1}}.
    \label{eq:nonfactHtildenonfact2}
\end{equation}
Taking the continuum limit $h\to0^{+}$ we have:
\begin{equation}
    \tilde{H} 
    =-1 - 4 \imath h +8 h^2+ \frac{4}{3}\imath(3H_{2} + 10)h^{3} + O\left( h^{4} \right),
    \label{eq:nonfactHtildeclim2}
\end{equation}
so we see that we recover, up to the addition of an inessential constant,
the continuum first integral \eqref{eq:nonfact2}.
Computing the limit we used that when $h\to0$ $\delta_{2}<0$.

Like in the previous example we can consider the singular fibres of the
pencil associated to $\tilde{H}$ \eqref{eq:nonfactHtildenonfact2}. 
These singular fibres are again union of three lines,
union of a line and a conic, and nodal cubics. So, the singular
fibres configuration is again of type $A_{2}^{2}\oplus A_{1}^{2}$,
i.e. number 40 from \cite[Table 8.2]{SchuttShioda2019mordell}. In this case we do
not present the explicit expression of these curves since it
is rather cumbersome and it does not add any further information.

To summarise, this example adds to the previous one the fact that there
exist cases when the ``hexagon'' formed by the base points is a non-convex
polygon, and that the base points can become complex in a neighbourhood
of zero. In Figure \ref{fig:nonfact2} we give a graphical representation
of this occurrence.

\begin{figure}[hbt]
    \centering
    \includegraphics{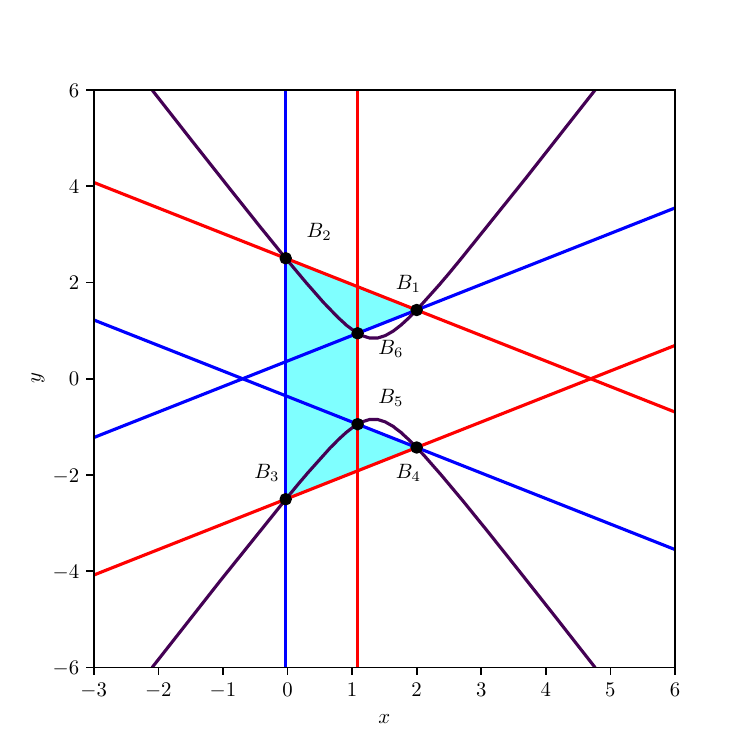}
    \caption{The non-factorisable case \eqref{eq:nonfact2} with 
        $h=7/10$: the lines $\overline{B_{1}B_{2}}$, $\overline{B_{3}B_{4}}$, 
        and $\overline{B_{5}B_{6}}$ in red, the lines $\overline{B_{2}B_{3}}$, 
        $\overline{B_{4}B_{5}}$, and $\overline{B_{6}B_{1}}$ in blue.
        The hyperbola \eqref{eq:nonconhyp} is displayed in purple, while the
non-convex hexagon formed by the base points is highlighted in cyan.}
    \label{fig:nonfact2}
\end{figure}

\subsection{The degenerate case: the conic curve}

The results of this section do not follow from the general results presented
in Theorem \ref{thm:triangles}, but rather form an extension for another
case. As it will be more clear later this might be a bridge for further
developments of the results presented in this paper to other cases of
interest.

In this example we consider the case when the Hamiltonian is the most general 
quadratic Hamiltonian:
\begin{equation}
    H = \frac{1}{2} a x^{2} + b xy+ \frac{1}{2} c y^{2} + d x + e y,
    \label{eq:Hq}
\end{equation}
where an additional constant term was omitted because it is inessential
for the equation of motion.
If the skew-symmetric matrix $J$ is constant, then the associated Hamiltonian
system is \emph{linear}.
However, we can impose that the associated Hamiltonian system is quadratic
by considering a system of the form:
\begin{equation}
    \vec{\dot{x}} = G(x,y) J \grad H(x,y).
    \label{eq:gaugeham}
\end{equation}
with an affine gauge function $G$.
In a similar way as it was discussed in \cite{GJ_biquadratic}, up to
canonical transformations we can always put $G=x$.
So, the associated Hamilton equations are:
\begin{equation}
    \dot{x} = x \left( b x + c y + e \right),
    \quad
    \dot{y} = - x \left( a x + b y + d \right).
    \label{eq:hequadratic}
\end{equation}

Following rule \eqref{eq:kahan} we construct the following discretisation:
\begin{subequations}
    \begin{align}
        \frac{x'-x}{h} &= \left(b x+ \frac{c}{2} y+\frac{e}{2}\right) x' 
        + \frac{c}{2} x y'+ \frac{e}{2} x, 
        \label{eq:hequadratickhka}
        \\
        \frac{y'-y}{h} &= - \left( a x+ \frac{b}{2} y+ \frac{d}{2}\right) x' 
        - \frac{b}{2} x y'- \frac{d}{2} x.
        \label{eq:hequadratickhkb}
    \end{align}
    \label{eq:hequadratickhk}%
\end{subequations}
An invariant can be constructed following \cite{CelledoniMcLachlanMcLarenOwrenQuispel2017}:
\begin{equation}
    \tilde{H} = \dfrac{H+\Delta_2 h^2 x^2/8}{1+\Delta_1 h^2 x^2/4},
    \quad
    \Delta_{1} = ac - b^{2},\,
    \Delta_{2} = -ae^2+2bde-cd^2,
    \label{eq:Htildeq}
\end{equation}
since Theorem \ref{thm:invariant} does not apply.
The associated pencil is:
\begin{equation}
    \begin{aligned}
        p(x,y;e_{0}:e_{1}) &= e_{0} \left( \frac{a}{2}  x^{2} + b xy + \frac{c}{2} y^{2} 
        + d x + e y + \frac{\Delta_2}{8} h^2 x^2\right) 
        \\
        &+ e_{1}\left( 1+\frac{\Delta_1}{4} h^2 x^2\right),
    \end{aligned}
    \label{eq:Hqpencil}
\end{equation}
and it has vanishing genus. That is, the curve $p=0$ is a conic
like the level surfaces of $H$.
From Definition \ref{def:singular} the singular points of the
pencil \eqref{eq:Hqpencil} are obtained for:
\begin{equation}
    \left[ e_{0}': e_{1}' \right]
    =
    \left[ 2 c h^{2} : e^{2}h^{2} -4 \right],
    \quad
    \left[ e_{0}'': e_{1}'' \right]
    =
    \left[ 2\Delta_{1}:-\Delta_{2} \right].
    \label{eq:Hqsing}
\end{equation}
In both cases the pencil factorises as follows:
\begin{subequations}
    \begin{align}
        p(x,y;e_{0}':e_{1}') &
        \begin{aligned}[t]
            =
            c \Delta_{1}
            &\ell\left(x,y;-\frac{b e h-c d h+2 b}{2c}, -\frac{e h+2}{c h}\right)
            \\
            \cdot &\ell\left(x,y; \frac{b e h-c d h-2 b}{2c}, -\frac{e h-2}{c h}\right),
        \end{aligned}
        \label{eq:pquadrsing1}
        \\
        p(x,y;e_{0}'':e_{1}'') &
        \begin{aligned}[t]
            =
        c\Delta_{1} 
        &\ell \left( x,y;\frac{-b+\sqrt{-\Delta_{1}}}{c},\frac{cd-\left(b-\sqrt{-\Delta_{1}}\right)e}{c\sqrt{-\Delta_{1}}} \right)
        \\
        \cdot &\ell \left( x,y;-\frac{b+\sqrt{-\Delta_{1}}}{c},\frac{cd-\left(b+\sqrt{-\Delta_{1}}\right)e}{c\sqrt{-\Delta_{1}}} \right).
        \end{aligned}
        \label{eq:pquadrsing2}
    \end{align}
    \label{eq:pquadrsing}
\end{subequations}

On the other hand the indeterminacy points of the map \eqref{eq:hequadratickhk} 
are:
\begin{subequations}
    \begin{align}
        B_{1} &=
        \left( -\frac{2}{h\sqrt{-\Delta_{1}}}, 
        \frac{2b - h\left( be-cd \right)}{ch\sqrt{-\Delta_{1}}}-\frac{eh-2}{c h} \right),
        \label{eq:B1conic}
        \\
        B_{2} &=
        \left( \frac{2}{h\sqrt{-\Delta_{1}}}, 
            -\frac{2b + h\left( be-cd \right)}{ch\sqrt{-\Delta_{1}}}- \frac{eh+2}{c h} \right),
        \label{eq:B2conic}
        \\
        B_{3} &=
        \left(-\frac{2}{h\sqrt{-\Delta_{1}}}, 
         \frac{2b + h\left( be-cd \right)}{ch\sqrt{-\Delta_{1}}}-\frac{eh+2}{c h}  \right),
        \label{eq:B3conic}
        \\
        B_{4} &=
        \left( \frac{2}{h\sqrt{-\Delta_{1}}}, 
        -\frac{2b - h\left( be-cd \right)}{ch\sqrt{-\Delta_{1}}}-\frac{eh-2}{c h} \right).
        \label{eq:B4conic}
    \end{align}
\end{subequations}
This is consistent with the general theory of conic pencils.
Considering the lines:
\begin{subequations}
    \begin{align}
        \overline{B_{1}B_{2}} &=
        \ell \left( x,y;-\frac{b+\sqrt{-\Delta_{1}}}{c},
            \frac{cd-\left(b+\sqrt{-\Delta_{1}}\right)e}{c\sqrt{-\Delta_{1}}} \right),
        \label{eq:B1B2conic}
        \\
        \overline{B_{2}B_{3}} &=
        \ell\left(x,y;-\frac{b e h-c d h+2 b}{2c}, -\frac{e h+2}{c h}\right),
        \label{eq:B2B3conic}
        \\
        \overline{B_{3}B_{4}} &=
        \ell \left( x,y;\frac{-b+\sqrt{-\Delta_{1}}}{c},\frac{cd-\left(b-\sqrt{-\Delta_{1}}\right)e}{c\sqrt{-\Delta_{1}}} \right),
        \label{eq:B3B4conic}
        \\
        \overline{B_{4}B_{1}} &=
        \ell\left(x,y; \frac{b e h-c d h-2 b}{2c}, -\frac{e h-2}{c h}\right),
        \label{eq:B4B1conic}
    \end{align}
    \label{eq:coniclines}
\end{subequations}
we can write down the invariant \eqref{eq:hequadratic} as:
\begin{equation}
    \widehat{H} = 
    \frac{\overline{B_1B_2}\, \overline{B_3B_4}}{\overline{B_2B_3}\, \overline{B_4B_1}}.
    \label{eq:Hhatconic}
\end{equation}
That is, in this case the invariant is \emph{expressible as the ratio of
four lines}. In this case the lines are not necessarily pairwise parallel,
as is evident from the expression of the angular coefficients of
the lines in \eqref{eq:coniclines}, as shown in Figure \ref{fig:conics}. 
Moreover, note that the invariant \eqref{eq:Hhatconic} is a multiple of 
\eqref{eq:Htildeq}:
\begin{equation}
    \widehat{H} = \frac{h^{2}c^{2}}{\Delta_{1}} \tilde{H}.
    \label{eq:Hconicprop}
\end{equation}
This implies that the continuum limit of
$\widehat{H}$ is the Hamiltonian $H$ \eqref{eq:Hq} at order $h^{2}$.

\begin{figure}[htb]
    \centering
    \includegraphics{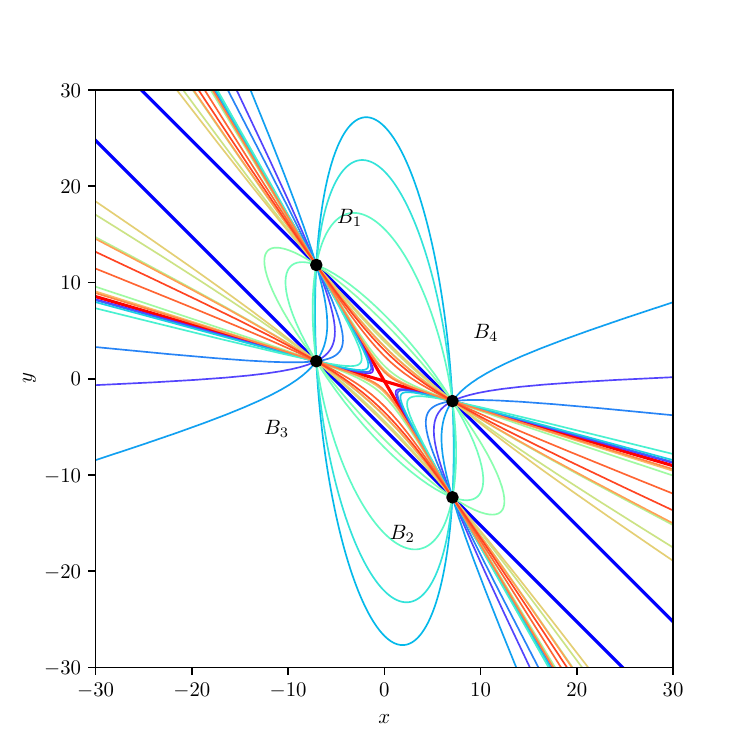}
    \caption{The pencil of conics \eqref{eq:Hqpencil} with 
        parameters $a=1$, $b =c = 2$, $d = e = 1/2$, and $h=1/5$. 
        In total 24 different combinations of $\left[ e_{0}:e_{1} \right]\in\Pj^{1}$
        are considered.
        The lines $\overline{B_{1}B_{2}}$, $\overline{B_{3}B_{4}}$ are in red,
        while the lines $\overline{B_{2}B_{3}}$, $\overline{B_{4}B_{1}}$ are in blue.}
    \label{fig:conics}
\end{figure}

\section{Conclusions}
\label{sec:conclusions}

In this paper we have shown how to construct in an elementary way the
invariant of the KHK discretisation of a two-dimensional Hamiltonian
system. This construction was possible because of the particular
structure of the KHK birational map, as highlighted in \cite{PSS2019},
and the concept of singular fibre of a pencil of curves. Our main
result, \Cref{thm:triangles}, tells us that such an invariant can
be written down as the product of the ratios of affine polynomials
defining the prolongation of the three parallel sides of a hexagon.
From this result, in \Cref{cor:singfib}, we identified the singular
fibre configuration of the generic KHK discretisation of a Hamiltonian
cubic system to be of type $A_{2}^{2}\oplus A_{1}$, i.e. number 20
from \cite[Table 8.2]{SchuttShioda2019mordell}.  Then, we noticed
that \Cref{thm:triangles} enables us to construct the invariant of a
KHK discretisation of a Hamiltonian cubic system simply by looking
at its indeterminacy points.  We presented several examples of this
construction. In particular, in those examples we observed that in some
cases the configuration of singular fibres is bigger than in the general
cases, presenting examples with singular fibres configuration of type
$A_{2}^{3}\oplus A_{1}$ (number 61) and $A_{2}^{2}\oplus A_{1}^{2}$
(number 40).  In the first case, the additional triple of lines allowed
the construction of multiple representations of the invariant in terms of
ratios of linearly factorised cubic polynomials.  Finally, we showed
an example of conic curves which is outside the hypotheses of Theorem
\ref{thm:triangles}, but where a similar final result is obtained.

The conic example is built using the ideas of
\cite{CelledoniMcLachlanMcLarenOwrenQuispel2017}, although the conic
case was not considered there.  That example belong to the class
of the discrete Nahm systems which are some of the most studied KHK
discretisation since their appearance in \cite{PetreraPfadlerSuris2009},
see for instance \cite{CelledoniMcLachlanMcLarenOwrenQuispel2017,
PetreraZander2017, GJ_biquadratic, GubNahm, Zander2021, GubShiNahm}
and their interpretation in terms of generalised Manin transform
in \cite{VanDerKampCelledoniMcLachlanMcLarenOwrenQuispel2019,
Petreraetal2021}.

We hope that our result will be useful in shedding light on why
integrability is preserved or not preserved by the KHK discretisation, and
we hope that it will be possible to extend our result to other known
integrable systems, both in the plane or in higher dimensions.
Regarding the last topic, we observe that recently some construction
of integrable systems in three dimensions where singular fibres play a
fundamental r\^ole appeared in the literature
\cite{GG_cremona3,Alonso_et_al2022}.

\section*{Acknowledgments}

This work was done within the framework of the Project ``Meccanica dei Sistemi
discreti'' of the GNFM unit of INDAM.\  In particular, GG acknowledges support
of the GNFM through Progetto Giovani GNFM 2023: ``Strutture variazionali e
applicazioni delle equazioni alle differenze
ordinarie''~(CUP\_E53C22001930001).

The figures in this paper are \texttt{eps} produced in \texttt{python}
using the libraries \texttt{numpy} \cite{Harris2020NumPy} and
\texttt{mathplotlib} \cite{Hunter2007}.

\appendix

\section{Explicit form of the coefficients in eq.\eqref{eq:CDexpl}}
\label{app:explicit}

Here are the formulas referred to in the statement of \Cref{lem:khkform}:
\begin{subequations}
    \begin{align}
        d_1 &= b_1 \mu_1 \mu_2-b_2 \mu_2 \mu_3+b_3 \mu_1 \mu_3-b_4 \mu_1 \mu_2+b_5 \mu_2 \mu_3-b_6 \mu_1 \mu_3, 
        \\
        d_2 &
        \begin{aligned}[t]
            &= -b_1 \mu_1-b_1 \mu_2+b_2 \mu_2+b_2 \mu_3-b_3 \mu_1-b_3 \mu_3+b_4 \mu_1
            \\
            &+b_4 \mu_2-b_5 \mu_2-b_5 \mu_3+b_6 \mu_1+b_6 \mu_3, 
        \end{aligned}
        \\
        d_3 &= b_1-b_2+b_3-b_4+b_5-b_6, 
        \\
        d_4 &= b_1 b_3 \mu_1+b_1 b_5 \mu_2-b_2 b_4 \mu_2-b_2 b_6 \mu_3+b_3 b_5 \mu_3-b_4 b_6 \mu_1, 
        \\
        d_5 &= -b_1 b_3-b_1 b_5+b_2 b_4+b_2 b_6-b_3 b_5+b_4 b_6,
        \\
        c_5 &= \frac{
            \left(
            \begin{gathered}
                b_1 b_2 b_3 b_5 \mu_2 \mu_3-b_1 b_2 b_4 b_6 \mu_1 \mu_2+b_1 b_3 b_4 b_5 \mu_1 \mu_2
                \\
                +b_1 b_3 b_5 b_6 \mu_1 \mu_3-b_2 b_3 b_4 b_6 \mu_1 \mu_3-b_2 b_4 b_5 b_6 \mu_2 \mu_3
            \end{gathered}
        \right)
    }{\Delta}, 
        \\
        c_6 &= -\frac{
            \left(
            \begin{gathered}
                b_1 b_2 b_3 b_5 \mu_2+b_1 b_2 b_3 b_5 \mu_3-b_1 b_2 b_4 b_6 \mu_1
                -b_1 b_2 b_4 b_6 \mu_2
                \\
                +b_1 b_3 b_4 b_5 \mu_1+b_1 b_3 b_4 b_5 \mu_2
                +b_1 b_3 b_5 b_6 \mu_1+b_1 b_3 b_5 b_6 \mu_3
                \\
                -b_2 b_3 b_4 b_6 \mu_1
                -b_2 b_3 b_4 b_6 \mu_3-b_2 b_4 b_5 b_6 \mu_2-b_2 b_4 b_5 b_6 \mu_3
            \end{gathered}
        \right)
        }{\Delta},
        \\
        c_7 &= \frac{
            \left(
            \begin{gathered}
                b_1 b_2 b_3 b_5-b_1 b_2 b_4 b_6+b_1 b_3 b_4 b_5
                \\
                +b_1 b_3 b_5 b_6-b_2 b_3 b_4 b_6-b_2 b_4 b_5 b_6
        \end{gathered}\right)
    }{\Delta}, 
        \\
        c_8 &= \frac{
            \left(
            \begin{gathered}
                b_1 b_2 b_3 b_4 b_5 \mu_2-b_1 b_2 b_3 b_4 b_6 \mu_1+b_1 b_2 b_3 b_5 b_6 \mu_3
                \\
                -b_1 b_2 b_4 b_5 b_6 \mu_2+b_1 b_3 b_4 b_5 b_6 \mu_1-b_2 b_3 b_4 b_5 b_6 \mu_3
        \end{gathered}\right)
    }{\Delta}, 
        \\
        c_9 &= -\frac{
            \left(
            \begin{gathered}
                b_1 b_2 b_3 b_4 b_5-b_1 b_2 b_3 b_4 b_6+b_1 b_2 b_3 b_5 b_6
                \\
                -b_1 b_2 b_4 b_5 b_6+b_1 b_3 b_4 b_5 b_6-b_2 b_3 b_4 b_5 b_6
        \end{gathered}\right)
    }{\Delta},
    \end{align}
    \label{eq:params}
\end{subequations}
and
\begin{equation}
    \Delta = b_2 b_4 b_6-b_1 b_3 b_5.
    \label{eq:Deltadef}
\end{equation}

\section{Explicit form of coefficients in equation \eqref{eq:h3eq}}
\label{app:map}

Here are the formulas referred to in \Cref{rem:darboux}:
\begin{subequations}
    \begin{align}
        a_1&= \frac{
        b_{25} b_{36}\mu_{12}^2 \mu_3^3
        -b_{14}b_{36}\mu_{23}^2 \mu_1^3
        + b_{14} b_{25}\mu_{13}^2 \mu_2^3
        }{hD}, 
        \label{eq:a1nmu}
        \\
        a_2&= \frac{
        -  b_{25} b_{36}\mu_{12}^2 \mu_3^2
        +b_{14}b_{36}\mu_{23}^2 \mu_1^2
        - b_{14} b_{25}\mu_{13}^2 \mu_2^2
        }{hD},
        \label{eq:a2nmu}
        \\
        a_3&= \frac{
        b_{25} b_{36}\mu_{12}^2 \mu_3
        -b_{14}b_{36}\mu_{23}^2 \mu_1
        + b_{14} b_{25}\mu_{13}^2 \mu_2
        }{hD},
        \label{eq:a3nmu}
        \\
        a_4&= \frac{
         - b_{25} b_{36}\mu_{12}^2
        +b_{14}b_{36}\mu_{23}^2 
        - b_{14} b_{25}\mu_{13}^2 
        }{hD},
        \label{eq:a4nmu}
        \\
        a_5&= \frac{
            \left\{
            \begin{gathered}
                (b_1+b_4) b_{25} b_{36} \mu_3^2 \mu_{12}^2
                - (b_2+b_5) b_{14} b_{36} \mu_1^2 \mu_{23}^2
                \\
                + (b_3+b_6) b_{14} b_{25} \mu_2^2 \mu_{13}^2
            \end{gathered}
        \right\}
        }{2hD},
        \label{eq:a5nmu}
        \\
        a_6&= \frac{
            \left\{
            \begin{gathered}
        - (b_1+b_4) b_{25} b_{36} \mu_3 \mu_{12}^2
        + (b_2+b_5) b_{14} b_{36} \mu_1 \mu_{23}^2
        \\
        - (b_3+b_6) b_{14} b_{25} \mu_2 \mu_{13}^2
\end{gathered}\right\}
        }{2hD},
        \label{eq:a6nmu}
        \\
        a_7&= \frac{
        (b_1+b_4) b_{25} b_{36}  \mu_{12}^2
        - (b_2+b_5) b_{14} b_{36}  \mu_{23}^2
        +(b_3+b_6) b_{14} b_{25}  \mu_{13}^2 }{2hD},
        \label{eq:a7nmu}
        \\
        a_8&= \frac{ 
        b_1 b_4 b_{25} b_{36} \mu_3 \mu_{12}^2
        - b_2 b_5 b_{14} b_{36} \mu_1 \mu_{23}^2
        + b_3 b_6 b_{14} b_{25}  \mu_2 \mu_{13}^2
        }{hD},
        \label{eq:a8nmu}
        \\
        a_9&= \frac{
        - b_1 b_4 b_{25} b_{36}  \mu_{12}^2
        + b_2 b_5 b_{14} b_{36}  \mu_{23}^2
        - b_3 b_6 b_{14} b_{25}  \mu_{13}^2
        }{hD},
        \label{eq:a9nmu}
    \end{align}
    \label{eq:abmu}
\end{subequations}
where
\begin{equation}
    D=\frac{1}{2} b_{14}b_{25}b_{36}\mu_{12}\mu_{13}\mu_{23}
    \label{eq:Ddef}
\end{equation}
and $b_{ij}=b_i-b_j$, $\mu_{ij}=\mu_i - \mu_j$. These formulas, with
$h=1$, were first presented in Appendix B in \cite{PSS2019}. We note that
since $a_{i}=O\left( 1 \right)$, the coefficients $b_{i}$ and $\mu_{i}$
depend on $h$.

\section{Explicit form of the polynomials in equation \eqref{eq:pprime}}
\label{app:PQ}

Here are the polynomials forming the cofactors of the Darboux 
polynomial in \Cref{rem:darboux}:
\begin{subequations}
    \begin{align}
        P_{1} &
        \begin{aligned}[t]
            &=
            \left[  
                \left(b_4\mu_2-b_3\mu_3\right)\mu_{12}  +\left(b_5 \mu_2 -b_6\mu_1\right)\mu_{23}  
            \right]x
            \\
            &+\left[  
                \left( b_3 - b_4\right)\mu_{12} - \left( b_5-  b_6\right) \mu_{23}
            \right]y
            +b_4 b_6 \mu_{12} - b_3 b_6\mu_{13} +b_3 b_5\mu_{23} ,
        \end{aligned}
        \label{eq:P1exp}
        \\
        P_{2} &
        \begin{aligned}[t]
            &=
            \left[  
                \left(b_1\mu_1 - b_2\mu_3 \right)\mu_{23} +\left(b_3\mu_3  - b_4\mu_2 \right) \mu_{13}
            \right]x
            \\
            &-
            \left[ 
                \left( b_1- b_2\right)\mu_{23} +\left( b_3- b_4\right)\mu_{13}
            \right]y
            -  b_1 b_4\mu_{12}
            +b_1 b_3\mu_{13} 
            - b_2 b_4\mu_{23},
        \end{aligned}
        \label{eq:P2exp}
        \\
        P_{3} &
        \begin{aligned}[t]
            &=
            \left[  
                \left(b_1\mu_1  -b_2\mu_3  \right) \mu_{12} +\left(b_5\mu_2 -b_6\mu_1\right) \mu_{13} 
            \right]x
            \\
            &-
            \left[ 
                \left(b_1- b_2\right)\mu_{12}+\left( b_5- b_6\right)\mu_{13}
            \right]y
            +b_1 b_5\mu_{12}-b_2 b_6\mu_{13}+b_2 b_5\mu_{23},
        \end{aligned}
        \label{eq:P3exp}
        \\
        Q &
        \begin{aligned}[t]
            &=
            \left(b_{14} \mu_1 \mu_2+b_{36} \mu_1 \mu_3-b_{25} \mu_2 \mu_3\right) x^2
            \\
            &-\left[b_{14} \left(\mu_1+ \mu_2\right) 
                -b_{25} \left(\mu_2+\mu_3\right)
                +b_{36} \left(\mu_1+\mu_3\right)\right] x y
            \\
            &+(b_{12}+b_{34}+b_{56}) y^2
            \\
            &+\left[\left(b_1 b_3 -b_4 b_6\right) \mu_1
            +\left(b_1 b_5 -b_2 b_4\right) \mu_2
            -\left(b_2 b_6 -b_3 b_5\right) \mu_3\right] x
            \\
            &-(b_1 b_3+b_1 b_5-b_2 b_4-b_2 b_6+b_3 b_5-b_4 b_6) y
            +b_1 b_3 b_5-b_2 b_4 b_6.
        \end{aligned}
    \end{align}
    \label{eq:PQ}
\end{subequations}

\printbibliography

\label{lastpage}
\end{document}